\documentclass[headsepline]{scrartcl} 
\usepackage{style_sheet}  

\begin{document}

\hyphenation{Ma-the-ma-tik ma-the-ma-tischen ana-lysis multi-pli-ca-tion de-riva-tive de-riva-tives trans-form-ation ap-proxi-mate eu-clid-ean de-ter-mine rep-re-sen-ta-tion rep-re-sen-ta-tions posi-tive per-pen-dicu-lar math-emat-ics com-pari-son de-ter-mined exam-ine exam-ined ex-peri-ment ex-peri-ments de-ter-min-ant de-ter-min-ants exact-ly re-sult re-sults es-ti-ma-tion con-eigen-val-ues}

\pdfbookmark[1]{Title}{title}
\title{The Generalized Operator Based Prony Method}
\author {Kilian Stampfer\footnote{Institute for Mathematical Stochastics,
G\"ottingen University, Lotzestr.\ 16-18,  37083 G\"ottingen, Germany. Email:
k.stampfer@math.uni-goettingen.de} \qquad Gerlind Plonka\footnote{Institute for
Numerical and Applied Mathematics, G\"ottingen University, Lotzestr.\ 16-18,  37083 G\"ottingen, Germany. Email: plonka@math.uni-goettingen.de}}

\maketitle  

\pdfbookmark[1]{Abstract}{abstract}
\abstract{{\normalfont \bfseries Abstract.} 
The generalized Prony method introduced in \cite{Peter.2013b} is a reconstruction technique for a large variety of sparse signal models that can be represented as sparse expansions
into eigenfunctions of a linear operator $A$. However, this procedure requires the evaluation of higher powers of the linear operator $A$ that are often expensive to provide. 

In this paper we propose two important extensions of the generalized Prony method that simplify the
acquisition of the needed samples essentially and at the same time can improve the numerical stability
of the method. The first extension regards the change of operators from $A$ to $\varphi(A)$,
where $\varphi$ is a suitable operator valued mapping, such that  $A$ 
and $\varphi(A)$ possess the same set of eigenfunctions.
The goal is now  to choose $\varphi$ such that the powers of $\varphi(A)$ are much simpler to evaluate than
the powers of $A$. The second extension concerns the choice of the sampling functionals. We show, how new 
sets of different sampling functionals $F_{k}$ can be applied with the goal to reduce the needed number of 
powers of the operator $A$ (resp. $\varphi(A)$) in the sampling scheme and to simplify the acquisition 
process for the recovery method.
\medskip


}

\section{Introduction}
\setcounter{equation}{0}

The recovery of signals  which can be represented or approximated  by finite expansions  into signal atoms is a task regularly
encountered in a variety of fields such as signal processing, biology, and
engineering. 
These ``signal atoms''  have a fixed structure and can be identified by a small number of real or complex parameters.
Therefore,  sparse expansions into these  signal atoms  often permit an arbitrarily high resolution
 in contrast to classical sampling schemes based on Hilbert space techniques.
At the same time these signal models frequently allow a better  physical interpretation.
The most prominent and well-studied signal model of this kind is  a sparse 
expansion into complex exponentials,
 i.e., 
\begin{equation}\label{expo0}
f(x):=\sum\limits_{j=1}^M c_j\exp(T_jx)=\sum\limits_{j=1}^Mc_j z_j^x,
\end{equation}
with pairwise different $z_{j} := \exp(T_{j})$ and with parameters $c_{j} \in {\mathbb C} \setminus  \{ 0 \}$ and $T_{j} \in {\mathbb C}$.
Using the classical Prony method, the parameters $c_{j}$ and $z_{j}$ can be computed from the $2M$ equidistant  samples $f(\ell)$, $\ell=0, \ldots , 2M-1$, see e.g.\ \cite{PT14} or \cite{PPST18}, Chapter 10, and the references therein. Observe that, in order to extract $T_{j}$ from $z_{j}$ in a unique way, we need to restrict $\text{Im}\, T_{j}$ to an interval of length $2\pi$.

In practical applications, we have to take special care of the numerical instabilities that can occur using Prony's method.
There have been many attempts to provide improved numerical algorithms, including the Pisarenko method \cite{Pisarenko.1973},
MUSIC \cite{Schmidt.1986}, ESPRIT
\cite{Kailath.1990}, Matrix Pencil Methods \cite{Hua.1991} and the approximate Prony method \cite{Potts.2010}.
Furthermore, to ensure the consistency in case of noisy measurements, modifications of Prony's method have been proposed, see  e.g.\ \cite{BM86, Osborne.1995, LiS00, ZP18}.
The interest in Prony-like methods has been strongly increased during the last years, also because of their utilization for the recovery of signals of finite rate of innovation,  see e.g.\ \cite{Vetterli.2002, Dragotti.2007, Urigen.2013, BE19}. In particular, the close connection between the exponential sum in (\ref{expo0}) and the expansion into shifted Diracs
\begin{equation} \label{spike} s(t) = \sum_{j=1}^{M} c_{j} \, \delta(t-t_{j}) 
\end{equation}
with $c_{j} \in {\mathbb C} \setminus \{ 0 \}$ and $t_{j} \in {\mathbb R}$
is extensively used. Indeed the Fourier transform of $s(t)$ is of the form (\ref{expo0}), where $T_{j} = {\mathrm i} t_{j}$, and thus $s(t)$ can be reconstructed from only $2M$ of its Fourier samples, see also \cite{PPT11, PW13}. 
 Moreover, Prony's method and its generalizations provide new approaches for nonlinear sparse approximation of smooth functions, and there are close relations to optimal approximation of functions in Hardy spaces \cite{ACH11, BM05, PP16, PP19}. 
\smallskip

An essential extension of the classical Prony method has been proposed in \cite{Peter.2013b}, where the recovery of expansions into exponentials has been generalized to
the recovery of expansions into  eigenfunctions of linear operators. 

Let us assume that $A: V \to V$ is a linear operator on a normed vector space $V$, and let $\sigma(A)$ be a subset of the point spectrum of $A$ that contains pairwise different eigenvalues.
Further, we consider the corresponding set of eigenfunctions $v_{\lambda}$  of $A$  such that $v_{\lambda}$ can be uniquely identified by $\lambda \in \sigma(A)$. In other words, the eigenspace to  $\lambda$ is fixed as a one-dimensional space.
Then, the generalized Prony method in  \cite{Peter.2013b} allows the reconstruction of expansions $f$ of the form 
\begin{equation}\label{expog1} f = \sum_{j=1}^{M} c_{j} \, v_{\lambda_{j}} 
\end{equation}
with $c_{j} \in {\mathbb C} \setminus \{ 0 \}$ and with pairwise distinct $\lambda_{j} \in \sigma(A)$. According to \cite{Peter.2013b}, the eigenvalues $\lambda_{j}$ belonging to the ``active'' eigenfunctions $v_{\lambda_{j}}$ and the coefficients $c_{j}$, $j=1, \ldots , M$
can be uniquely recovered from the (complex) values $F(A^{\ell}f)$, $\ell=0, \ldots , 2M-1$, where $F:V \to {\mathbb C}$ is a functional that can be chosen arbitrarily up to the condition $F v_{\lambda} \neq 0$ for all $\lambda \in \sigma(A)$. 
The expansion into exponentials in (\ref{expo0}) can be seen as a special case of  (\ref{expog1}) if we take $V= C({\mathbb R})$, $A=S_{1}$ with the shift operator given by  $S_{1} f := f(\cdot +1)$, and the point evaluation functional $F f := f(0)$. Indeed, the exponentials $\exp(T_{j} \cdot)$ are eigenfunctions of $S_{1}$ to the eigenvalues $\exp(T_{j})$ which are pairwise different for $T_{j} \in {\mathbb R} + {\mathrm i} [-\pi, \pi)$. The needed samples $F(A^{\ell}f)$ are in this case of the form $F(A^{\ell}f)= F(S_{1}^{\ell}f) = F(f(\cdot + \ell)) = f(\ell)$.

There have been other attempts to generalize the idea of Prony's method to different expansions, including sparse polynomials \cite{BT88}, piecewise sinusoidal signals \cite{BDB10}, sparse expansions into Legendre polynomials \cite{PPR13} or Chebyshev polynomials \cite{Potts.2014} and into Lorentzians \cite{BSBV17}.
All these expansions can be also recovered directly using the approach in \cite{Peter.2013b}. 
An  extension of the generalized Prony method to the multivariate case based on Artinian Gorenstein algebras and the flat extension principle has been given by Mourrain \cite{Mourrain.2017}.
\medskip

However, the generalized Prony method is not always simple to apply since it requires the computation of higher powers of the operator $A$ in order to achieve the needed sample values $F(A^{\ell}f)$ for the reconstruction procedure. While for shift operators these samples are easy to acquire, the problem is much more delicate for differential or integral operators of higher order.
Indeed, the shift operator $S_{\tau}$, with $S_{\tau} f := f( \cdot + \tau)$, and its generalizations play a special role, since the power $S_{\tau}^{\ell}$ is  equivalent to  $S_{\ell \tau}$, i.e., to a simple shift operator with  shift length $\ell \tau$. Expansions into eigenfunctions of generalized shift operators are therefore of special interest, since they can be recovered just by suitable function samples, see \cite{PSK18}.
\smallskip

%
In this paper, we reconsider the generalized Prony method in \cite{Peter.2013b} in more detail and particularly study two extensions that provide us more freedom in data acquisition for the recovery of expansions of the form (\ref{expog1}).

The first extension is based on the observation that for a given linear operator $A$ there is often a different linear operator $B$ that possesses the same eigenfunctions to different eigenvalues.
For example, the exponential function $\exp(Tx)$ is an eigenfunction of the shift operator $S_{\tau}$  to the eigenvalue ${\mathrm e}^{\tau T}$, 
but at the same time also an eigenfunction of the differential operator $\frac{\mathrm d}{{\mathrm d} x}$  to the eigenvalue $T$.
 Thus, we need to understand, how this observation can help us to solve the signal recovery problem,
 and in particular, for a given linear operator A, how to find a different linear operator 
$B$  with the same eigenfunctions that may be easier to apply.

The second extension directly aims at generalizing the sampling functional $F$. While it is appealing that the $2M$ parameters of the signal model in (\ref{expo0}) and (\ref{expog1}) can be theoretically obtained from only $2M$ samples, in many applications we are faced with a parameter identification problem, where a large number of noisy samples is given, and we need to identify the parameters in a stable manner.
Therefore, we go away from sampling schemes that use a minimal number of sampling values being ordered in matrices with Hankel structure.
We will show that there is much more freedom to choose a set of different sampling functionals $F_{k}$, where each sampling functional leads to a linear equation providing one condition for the vector of coefficients of the Prony polynomial. 
Our approach also covers previous ideas to identify the frequency parameters $T_{j}$ of the exponential sum in (\ref{expo0}) using equispaced sampling sequences with different sampling sizes simultaneously, see \cite{CL18}.
\smallskip

Our ideas to provide simple acquisition schemes to recover expansions into eigenfunctions of linear operators also open the way for new approaches for sparse nonlinear approximation of (non-stationary) signals and images.

The paper is organized as follows. In Section \ref{sec_GPM} we reconsider the Prony method for exponential sums. We first show, how it can be understood as a method to recover a sparse expansion into eigenfunctions of the shift operator $S_{\tau}$ on the one hand and of the differential operator $\frac{\mathrm d}{{\mathrm d} x}$ on the other hand. In Section 2.3 we employ an exponential operator notation to show how the two operators $S_{\tau}$  and $\frac{\mathrm d}{{\mathrm d} x}$ are related to each other. Further, we introduce the idea, how the sampling scheme can be generalized using a set of different sampling functionals $F_{k}$ instead of $F(A^{k})$.

Section 3 is devoted to the new  generalized operator based Prony method (GOP).  We start with recalling the generalized Prony method from \cite{Peter.2013b} and transfer it into our new notation.
Sections 3.2 and 3.3 are concerned with  the two new extensions, first the change of operators from $A$ to $\varphi(A)$, where $\varphi$ is an analytic function, and second the generalization of the sampling scheme. 
In particular, we introduce admissible sets of sampling functionals $F_{k}$ that allow a unique reconstruction of expansions of the form (\ref{expog1}). 
In Section 3.4 we give a detailed example,  where  GOP is applied to sparse cosine expansions.

In Section 4, we discuss the application of  GOP for the recovery of eigenfunctions of differential operators. We show that special linear differential operators of first and second order lead by a transfer from the operator $A$ to  $\varphi(A)$ (with an exponential map $\varphi$) to generalized shift operators whose powers can be simply evaluated in sampling schemes.

Section 5 is devoted to a further investigation of the second extension, the generalized sampling.
We embed the functions $f$ in (\ref{expog1}) into a suitable Hilbert space and employ a dual approach for the sampling scheme. Then, our sampling functionals $F_k: V \to {\mathbb C}$ can be written as inner products with special kernels $\phi_{k}$ as Riesz representers, i.e., $ F_{k} (f) = \langle f , \phi_{k} \rangle$. Therefore the application of $F_{k}$ to powers $A^{\ell}f$ or $(\varphi(A))^{\ell} f$ to obtain the required sampling values can be rewritten by applying powers of the adjoint operator $A^{*}$ to the kernel $\phi_{k}$.  
In this way, we are able to find admissible sampling schemes for the recovery of expansions into eigenfunctions of differential operators in terms of moments.
We demonstrate the principle for the recovery of exponential sums and for the recovery of sparse Legendre expansions using only moments of $f$. 

The considerations in this paper provide the starting point for further studies that focus on the improvement of the numerical stability of the generalized Prony method. But this problem is beyond the scope of this paper and will be the further investigated.

\section{An introductory example: Revisiting Prony's method using shift and differential operator}
\label{sec_GPM}
\setcounter{equation}{0}

\subsection{Prony's method based on the shift operator}
\label{subsecshift}

The classical Prony method  is a way to reconstruct  the parameters  $c_{j} \in {\mathbb C} \setminus \{ 0 \}, \, T_{j} \in {\mathbb C}$, $j=1, \ldots , M$,  of  the weighted sum of exponentials 
\begin{equation}\label{expo} f(x) = \sum_{j=1}^{M} c_{j} \, \exp(T_{j}x).
\end{equation}
Using equidistant sample values $f(k)$, $k=0, \ldots , 2M-1$,  exact recovery  is possible if $T_{j} \in  {\mathbb R} + {\mathrm i} [-\pi, \, \pi)$, see e.g. \cite{PT14}.
Usually, we assume that there is an a priori known bound $C$ such that 
$\text{Im} \, T_{j} \in [-C\pi, \, C\pi)$, and the parameters $T_{j}$ can still be recovered using 
a rescaling argument and taking sampling values $f(kh)$ with  $h \le 1/C$ instead of $h=1$.
With 
\begin{align*}
\mathcal{M}:=\left\{\sum\limits_{j=1}^Mc_j\,e^{T_jx}: \, M<\infty,
c_j \in {\mathbb C} \setminus \{ 0 \},  T_j \in {\mathbb R} + {\mathrm i} [-C\pi, \, C\pi),  \forall j\neq i:T_j\neq T_i,\right\}
\end{align*}
we denote the model class of all finite linear combinations of complex exponentials that can be recovered by Prony's method.

Recalling the ideas in \cite{Peter.2013b,PSK18}, we can reinterpret and generalize the method using a shift operator.
The exponential sum in (\ref{expo}) can be understood as an expansion into $M$ eigenfunctions of the shift operator $S_{\tau}: C({\mathbb R}) \to C({\mathbb R})$ for some $\tau \neq 0$ with $S_{\tau} f (x) := f(x+\tau)$.
More precisely, we observe that 
$$ (S_{\tau} \exp(T_{j} \cdot))(x) = \exp(T_{j}(x + \tau)) = \exp(T_{j}\tau ) \, \exp(T_{j} x), $$
i.e., the exponentials $\exp(T_{j}x)$   occurring in (\ref{expo})  are eigenfunctions of $S_{\tau}$ to the eigenvalues $\exp(T_{j}\tau)$. This implies 
$$ (S_{\tau} -\exp(T_{j} \tau) I) \exp(T_{j} \cdot) = 0, $$
where $I$ denotes the identity operator. We define the Prony polynomial 
$$
 P(z) = P_{\tau}(z)  := \prod_{j=1}^{M} (z- \exp(T_{j}\tau)) 
$$
with the monomial representation
$$ P(z) = \sum_{\ell=0}^{M} p_{\ell} z^{\ell} = z^{M} + \sum_{\ell=0}^{M-1} p_{\ell} \, z^{\ell} $$
and observe for $f$ in (\ref{expo}) that 
\begin{eqnarray*}
P(S_{\tau}) f &=& \sum_{\ell=0}^{M} p_{\ell} S_{\tau}^{\ell} f
= \sum_{\ell=0}^{M} p_{\ell} S_{\tau}^{\ell} \sum_{j=1}^{M} c_{j} \exp(T_{j} \cdot) \\
&=& \sum_{j=1}^{M} c_{j} \, \exp(T_{j} \cdot) \Big( \sum_{\ell=0}^{M} p_{\ell} \, \exp(T_{j}\tau  \ell) \Big) \\
&=& \sum_{j=1}^{M} c_{j} \, \exp(T_{j} \cdot) \, P(\exp(T_{j}\tau)) = 0. 
\end{eqnarray*}
Thus, $f$ solves the difference equation $P(S_{\tau}) f = 0$.
In particular, we also have
$$ S_{\tau}^{k} \, P(S_{\tau}) f = P(S_{\tau}) \, S_{\tau}^{k} \, f =  \sum_{\ell=0}^{M} p_{\ell } \, S_{\tau}^{\ell+k}f =0, \qquad k \in {\mathbb Z}. $$
We fix an arbitrary value $x_{0} \in {\mathbb R}$ and employ the point evaluation functional $F_{x_{0}}$ with $F_{x_{0}} f := f(x_{0})$ to compute the  samples $F_{x_{0}} S_{\tau}^{k} f = f(x_{0}+ \tau k)$, $k=0, \ldots , 2M-1$.
Then we obtain the homogeneous equation system 
\begin{equation}
\label{2.s1} F_{x_{0}} ( S_{\tau}^{k} \, P(S_{\tau}) f)  = \sum_{\ell=0}^{M} p_{\ell} f(x_{0}+ \tau(k+\ell)) = 0, \quad k=0, \ldots , M-1,
\end{equation}
for the vector ${\mathbf p} = (p_{0}, \ldots , p_{M})^{T}$ of coefficients of $P(z)$.
For $f \in {\mathcal M}$ and fixed $\tau < C^{-1}$ the arising coefficient matrix $( f(x_{0}+\tau(k+ \ell)))_{k=0,\ell=0}^{M-1,M} \in {\mathbb C}^{M \times M+1}$ is of Hankel structure and has full rank $M$, see  \cite{Peter.2013b,PT14}. Thus, ${\mathbf p}$ is uniquely defined with $p_{M}=1$, and we can extract the zeros $\exp(T_{j}\tau)$ of the polynomial $P(z)$ and compute $T_{j}$, $j=1, \ldots , M$. Finally, the vector of coefficients ${\mathbf c} =(c_{j})_{j=1}^{M}$ in (\ref{expo}) can be computed as a least squares solution of  the Vandermonde system
$$ {\mathbf V}_{2M,M} \, {\mathbf c} = (S_{\tau}^{k} f(x_{0}))_{k=0}^{2M-1}
= (f(x_{0}+\tau k))_{k=0}^{2M-1} $$
with ${\mathbf V}_{2M,M} := (\exp(T_{j}(x_{0}+\tau k))_{k=0, j=1}^{2M-1, M}$.

\subsection{Prony's method based on the differential operator}
\label{subsecdif}

We now present a different  viewpoint  and interpret $f(x)$ in (\ref{expo}) as the solution of  a linear  ordinary differential equation of order $M$.
In fact, the functions $\exp(T_{j}x)$ are also eigenfunctions of the  first derivative operator $ \frac{{\mathrm d}}{{\mathrm d} x}: C^{\infty}({\mathbb R}) \to C^{\infty}({\mathbb R})$, i.e., 
$$ \left(\frac{{\mathrm d}}{{\mathrm d} x} \, \exp(T_{j} \cdot) \right) (x)  = T_{j} \, \exp(T_{j}x), $$
and thus 
$$
\left(\frac{{\mathrm d}}{{\mathrm d} x} - T_j I\right) \, \exp(T_j \cdot)=0
$$
for all $T_{j} \in {\mathbb C}$, where $I$ denotes the identity operator.
We can now proceed similarly as before just by replacing the shift operator with the differential operator.
Employing the eigenvalues $T_{j}$, we define the characteristic polynomial 
$$
\tilde{P}(z) :=  \prod_{j=1}^{M} ( z-T_{j}) 
= \sum_{\ell=0}^{M} \tilde{p}_{\ell} \, z^{\ell} = z^{M} + \sum_{\ell=0}^{M-1} \tilde{p}_{\ell} \, z^{\ell}.
$$
We apply the corresponding linear differential operator $\tilde{P}\left(\frac{{\mathrm d}}{{\mathrm d} x}\right)$ of order $M$ 
 to the function $f$ in (\ref{expo}) and find 
\begin{eqnarray*}
\tilde{P}\left(\frac{\mathrm d}{{\mathrm d}x}\right) f &=& \prod_{j=1}^{M} \Big( \frac{\mathrm d}{{\mathrm d}x} - T_{j}I \Big)\, f 
= \Big( \sum_{\ell=0}^{M} \tilde{p}_{\ell} \, \frac{{\mathrm d}^{\ell}}{{\mathrm d}x^{\ell}} \Big)\, f \\
&=& \sum_{\ell=0}^{M} \tilde{p}_{\ell} \, \sum_{j=1}^{M} c_{j} \, T_{j}^{\ell} \, \exp(T_{j} \cdot) = \sum_{j=1}^{M} c_{j} \, \exp(T_{j} \cdot) \Big( \sum_{\ell=0}^{M} \tilde{p}_{\ell} \, T_{j}^{\ell} \Big) \\
&=&  \sum_{j=1}^{M} c_{j} \, \exp(T_{j} \cdot) \, \tilde{P}(T_{j}) = 0, 
\end{eqnarray*}
i.e., $f$ in (\ref{expo}) solves the homogeneous differential equation $\tilde{P}\left(\frac{\mathrm d}{{\mathrm d}x}\right) f= 0$.
We particularly observe that 
$$
\frac{{\mathrm d}^k}{{\mathrm d} x^k} \tilde{P}\left(\frac{\mathrm d}{{\mathrm d}x}\right)f= \tilde{P}\left(\frac{\mathrm d}{{\mathrm d}x}\right) \, f^{(k)} =0
$$
for all $k \in {\mathbb N}$, where $f^{(k)}$ denotes the $k$-th derivative of $f$.
As before, we can exploit this observation in order to reconstruct the parameters $c_{j}$ and $T_{j}$, $j=1, \ldots , M$, that identify $f$.
We fix a value $x_{0} \in {\mathbb R}$ and
apply the point evaluation functional $F_{x_{0}}$ with $F_{x_{0}} f = f(x_{0})$ to obtain the equations  
\begin{equation}\label{2.9} F_{x_{0}} \left( \frac{{\mathrm d}^k}{{\mathrm d} x^k} \tilde{P}\left(\frac{\mathrm d}{{\mathrm d}x}\right)f \right)= F_{x_{0}} \left( \tilde{P}\left(\frac{\mathrm d}{{\mathrm d}x}\right) \, f^{(k)} \right) = \sum_{\ell=0}^{M} \tilde{p}_{\ell} \, f^{(k+\ell)}(x_{0}) =0,
\end{equation}
for $k=0, \ldots, M-1$. This homogeneous linear equation system yields the vector $\tilde{\mathbf p}=(\tilde{p}_{0}, \ldots, \tilde{p}_{M})$ of coefficients of the Prony polynomial $\tilde{P}(z)$. Also here, the arising Hankel matrix $(f^{(k+\ell)}(x_{0}))_{k=0,\ell=0}^{M-1,M}$ has full rank $M$, such that $\tilde{\mathbf p}$ is uniquely defined with $\tilde{p}_{M}=1$, see \cite{Peter.2013b}.
In turn we find the zeros $T_{j}$, $j=1, \ldots , M$, of $\tilde{P}(z)$. Now the coefficients $c_{j}$ can be obtained by solving the overdetermined linear system 
$$ \sum_{j=1}^{M} c_{j} T_{j}^{k} \, \exp(T_{j}x_{0}) = f^{(k)}(x_{0}), \qquad k=0, \ldots , 2M-1.$$

\subsection{Generalization 1: Switch between operators with the same eigenfunctions}
\label{subsec:gen1}

 An essential difference between the two approaches is that the required input values have completely different structure. Instead of the derivative values $f^{(k)}(x_{0})$ for some $x_{0} \in {\mathbb R}$ and $k=0, \ldots , 2M-1$ for $\frac{\mathrm d}{{\mathrm d}x}$, we  just need to provide the function values $f(x_{0}+ k\tau)$, $k=0, \ldots , 2M-1$ for $S_{\tau}$.

The second essential difference regards the condition of the matrices involved into the method.
For $ \frac{\mathrm d}{{\mathrm d}x}$ we have to find the zero eigenvector of the Hankel matrix $\tilde{\mathbf H} = (f^{(k+\ell)}(x_{0}))_{k=0,\ell=0}^{M-1,M} \in {\mathbb C}^{M \times M+1}$. Using the structure of $f(x)$ in (\ref{expo}), $\tilde{\mathbf H}$ has the factorization
$$ \tilde{\mathbf H} = \tilde{\mathbf V}_{M,M} \, \text{diag} ( c_{1}, \ldots , c_{M}) \,  \text{diag} ( \exp(T_{1}x_{0}), \ldots , \exp(T_{M} x_{0})) \, \tilde{\mathbf V}_{M+1,M}^{T},
$$
with the Vandermonde matrices $\tilde{\mathbf V}_{M,M} = (T_{j}^{\ell})_{\ell=0,j=1}^{M-1,M}$ and $\tilde{\mathbf V}_{M+1,M} = (T_{j}^{\ell})_{\ell=0,j=1}^{M,M}$.
In contrast, for $S_{\tau}$ we have instead to solve the eigenvalue problem with the Hankel matrix ${\mathbf H} = (f(x_{0}+\tau(k+\ell)))_{k=0,\ell=0}^{M-1,M}$ with the factorization 
$$ {\mathbf H} = {\mathbf V}_{M,M} \, \text{diag} ( c_{1}, \ldots , c_{M}) \,  \text{diag} ( \exp(T_{1}x_{0}), \ldots , \exp(T_{M} x_{0})) \, {\mathbf V}_{M+1,M}^{T},
$$
where ${\mathbf V}_{M,M}= (\exp(T_{j} \tau \ell))_{\ell=0,j=1}^{M-1,M}$ and ${\mathbf V}_{M+1,M}= (\exp(T_{j} \tau \ell))_{\ell=0,j=1}^{M,M}$.
Depending on the range of the parameters $T_{j}$ the occurring Vandermonde matrices
can have completely different condition number. If e.g.\ $T_{j} = {\mathrm i} \, \text{Im}\, T_{j}$, then the knots $\exp(T_{j}\tau)$ determining ${\mathbf V}_{M,M}$ lie on the unit circle while the
$T_{j}$ determining $\tilde{\mathbf V}_{M,M}$ lie on the imaginary axis.

We are therefore interested  in understanding the connection between the two methods to recover (\ref{expo}).
Both approaches work, since the exponentials $\exp(T_{j} x)$ are eigenfunctions to the two different operators 
$S_{\tau}$ and $\frac{{\mathrm d}}{{\mathrm d} x}$. But  the corresponding spectra 
are different. While the eigenvalues with regard to the differential operator 
$\frac{{\mathrm d}}{{\mathrm d} x}$  are of the form $T_{j}$, for the shift operator $S_{\tau}$ the eigenvalues are $\exp(T_{j}\tau)$. 
Obviously, the spectra are connected by the map $\exp(\tau \cdot): \, \lambda  \to \exp(\lambda \tau)$. 
With $\exp z = \sum\limits_{k=0}^{\infty} \frac{z^{k}}{k!}$ we indeed have 
\begin{align}
\nonumber
\exp{\left(\tau\frac{\mathrm d}{{\mathrm d}x}\right)} \exp(Tx)
&=\sum\limits_{k=0}^{\infty}\frac{\tau^k}{k!}\frac{{\mathrm d}^k}{{\mathrm d}x^k} \, \exp(Tx)
 =\sum\limits_{k=0}^{m}\frac{\tau^k}{k!} \, T^{k} \, \exp(Tx)
  \\
 \label{th_shift_exponential}
&= \exp(\tau T) \, \exp(Tx) =S_{\tau} \, \exp(Tx)
\end{align}
for all  $T \in {\mathbb C}$, and in turn for any analytic function $f\in {\mathcal M}$ 
$$\exp\left(\tau\frac{\mathrm d}{{\mathrm d}x}\right)f(x)=f(\tau+x) = (S_{\tau} \, f)(x), $$
see \cite{Dattoli.1997}.
Thus, using the analytic function $\exp(\tau \cdot)$, we can map from the differential operator $\frac{\mathrm d}{{\mathrm d}x}$ to the shift operator $S_{\tau}$, thereby staying with the same eigenfunctions but changing the eigenvalues.
This observation is summarized in the following Theorem.

\begin{theorem}
\label{th_spectral_mappingex}
Let  $\frac{\mathrm d}{{\mathrm d}x}: C^{1}({\mathbb R}) \to C({\mathbb R})$ be the  first derivative operator. Then, each $T\in {\mathbb C}$  is an eigenvalue of $\frac{\mathrm d}{{\mathrm d}x}$.      
For some $C >0$ let  $\Lambda_{C}:= {\mathbb R} + {\mathrm i} \, [-C \pi, \, C \pi)$ be a given subset of ${\mathbb C}$, and let $\varphi_{\tau}(x):=\exp(\tau x)$ with $\tau \le C^{-1}$.
Then $\varphi_{\tau}$ is well-defined on ${\mathbb C}$ and 
$$ \left( \frac{\mathrm d}{{\mathrm d}x} - T \, I \right) \, \exp(T  x) = 0 $$
implies 
$$ \varphi_{\tau} \left( \frac{\mathrm d}{{\mathrm d}x} \right) \exp(Tx)  - \varphi_{\tau}\left( T \, I \right) \exp(Tx) =  (S_{\tau} - \exp( \tau \, T) \, I) \, \exp(T  x) =0, $$
where $S_{\tau}$ is the shift operator as before.
Furthermore, the map $\varphi_{\tau}: T \to \exp(\tau \, T)$ is injective on $\Lambda_{C}$.
\end{theorem}

\begin{proof}
Obviously, $\frac{\mathrm d}{{\mathrm d}x} \, \exp(T x) = T \, \exp(T  x)$ for all $T \in {\mathbb C}$.
For all $T \in \Lambda_{C}$  the value $\varphi_{\tau}(T) = \exp( \tau \, T)$ is well-defined, and $\varphi_{\tau}(T_{1})=  \varphi_{\tau}(T_{2})$ yields
$T_{1} = T_{2} + \frac{2 \pi k {\mathrm i}}{\tau}$, $k \in {\mathbb Z}$, i.e., $T_{1} = T_{2}$ for $T_{1}, \, T_{2} \in \Lambda_{C}$.
The remaining assertions follow from (\ref{th_shift_exponential}). 
\end{proof}
\medskip

Theorem \ref{th_spectral_mappingex} has strong  implications on the reconstruction of $f(x)$ in (\ref{expo}) using Prony's method.  We can  
replace the operator $\frac{\mathrm d}{{\mathrm d}x}$ by the operator $S_{\tau}$ in order to reconstruct $f$ in (\ref{expo}), as we have seen in the previous two subsections.

\subsection{Generalization 2: Changing the sampling scheme}
\label{subsec:gen2}

In the two previous examples  in Subsections \ref{subsecshift} and \ref{subsecdif} we have applied the point evaluation functional  $F_{x_{0}}$ with some $x_{0} \in {\mathbb R}$ and used the samples 
$$ F_{x_{0}} (S_{\tau}^{k} \, f) = f(x_{0}+ k \tau) \quad \text{and} \quad F_{x_{0}} \left( \frac{\mathrm d^{k}}{{\mathrm d}x^{k}} f \right) = f^{(k)} (x_{0}), \quad k=0, \ldots , 2M-1, $$
respectively, to recover $f \in {\mathcal M}$.  According to \cite{Peter.2013b}, we can however  use any other linear functional $F: C^{\infty} \to {\mathbb C}$ with the only restriction that $F$ applied to the eigenfunctions $\exp(Tx)$ should be well-defined and nonzero for all $T$ in the parameter range we are interested in.
We can for example take
$$ F f = \int_{\Omega} f(x) \, K(x) \, {\mathrm d} x $$ 
with some $\Omega \subset {\mathbb R}$ and some rather arbitrary kernel function $K(x)$
such that $F f$ is  well defined and $\int_{\Omega} \exp(Tx) \, K(x) \, {\mathrm d} x  \neq 0$ for all $T \in {\mathbb C}$.   Thus, the choice of $F$ gives us already some freedom to choose the sampling scheme.
Taking e.g.\ $K(x) = \sum_{r=-L}^{L} w_{r} \delta(x - r\tau)$ with the delta distribution $\delta$ and some positive weights $w_{r}$ or just $K(x) := \chi_{[-1/2,1/2)}(x)$ we arrive at 
smoothed sampling values 
$$ F(S_{\tau}^{k} f) = \sum_{r=-L}^{L} w_{r} \, f((k+r)\tau) \quad \text{or}  \quad
F(S_{\tau}^{k} f) = \int_{-1/2}^{1/2} f(x+\tau k) \, {\mathrm d}x  $$
instead of $f(x_{0}+\tau k)$ for $ k=0, \ldots , 2M-1$.
\smallskip

We can now generalize the sampling scheme even further if we allow ourselves to employ more than the minimal number of $2M$ input data.
We inspect again the equations
$$ F_{x_{0}} ( S_{\tau}^{k} \, P(S_{\tau}) f) =0, \qquad k=0, \ldots , M-1,$$
that lead in (\ref{2.s1}) to the Hankel system determining the coefficient vector ${\mathbf p}$ of the Prony polynomial $P(z)$. 
We already have $P(S_{\tau}) f = 0$, and the application of $S_{\tau}^{k}$  does not change the right-hand side of the equation.
Therefore, for each $k=0, \ldots , M-1$, we can replace $F_{x_{0}} S_{\tau}^{k}$ by a new linear functional $F_{k}$ to obtain the $M$ equations to recover ${\mathbf p}$. 
We only need to pay attention that the obtained $M$ equations are linearly independent.
\medskip

For example, we could take $F_{k} = F_{x_{0}}\, S_{\theta}^{k}$ with a parameter $\theta \not\in \{0, \, \tau \}$ and obtain an equation system 
$$ F_{x_{0}} (S_{\theta}^{k} \, P(S_{\tau}) f) = \sum_{\ell=0}^{M} p_{\ell} f(x_{0}+ k\theta+ \ell\tau) = 0, \quad k=0, \ldots , M-1. $$
The arising coefficient matrix $(f(x_{0} + k \theta + \ell\tau))_{k=0,\ell=0}^{M-1,M}$  does not longer  have Hankel structure  but may  possess a  better condition  than $(f(x_{0}+ (k+\ell)\tau))_{k=0,\ell=0}^{M-1,M}$. Taking e.g.\  $\theta = 2 \tau$ we need the $3M-1$ sample values $f(x_{0}+\tau(2k+\ell))$ to recover $f$ in (\ref{expo}).

\smallskip

Considering the method in Section \ref{subsecdif}, we can also replace the  functionals $F_{x_{0}} \frac{\mathrm d^{k}}{{\mathrm d}x^{k}}$ in (\ref{2.9}) by other linear functionals $F_{k}$. Taking for example 
 $F_{k}= F_{x_{0}} \, S_{\tau}^{k}$ then we obtain the system
$$ F_{x_{0}} \left( S_{\tau}^{k} \tilde{P} \left(\frac{\mathrm d}{{\mathrm d}x} \right) f \right) = \sum_{\ell=0}^{M} p_{\ell} \, f^{(\ell)}(x_{0} + \tau k) =0, \qquad k=0, \ldots , M-1. $$
Here, we need now the input data $f^{(\ell)}(x_{0}+ k \tau)$, $k=0, \ldots , M-1, \, \ell=0, \ldots , M$, using only derivatives up to order $M$ and its equidistant shifts.
In Section \ref{sec3.3} and in Section \ref{sec_generalized-sampling}  we will investigate such generalized sampling schemes in more detail and particularly show that the examples above provide sampling matrices of full rank $M$, such that $f$ in (\ref{expo}) can be uniquely reconstructed. 

\begin{remark}
1. Special generalized sampling schemes for the shift operator and the differential operator have also been proposed by Seelamantula \cite{Seelamantula.2015}, but without considering the relations between these operators.  However, a rigorous investigation of rank properties of the involved matrices  has not been given in \cite{Seelamantula.2015}. The representation of Prony's method as an approach to reconstruct expansions into eigenfunctions of linear operators has been given already in \cite{Peter.2013b}.

2. For the special case  of recovery of  expansions into shifted Diracs  in (\ref{spike}), it has been  extensively studied  how to retrieve  the needed Fourier samples from  low-pass projections  with suitable sampling kernels, see e.g. \cite{Vetterli.2002, Dragotti.2007, BDB10, Urigen.2013, BSBV17, BE19}.

\end{remark}

\section{Generalized operator based Prony method}
\label{sec_GoProm}
\setcounter{equation}{0}

We  want to study the two new observations considered for the special operators $\frac{\mathrm d}{{\mathrm d}x}$ and $S_{\tau}$ in Subsections \ref{subsec:gen1} and \ref{subsec:gen2} in a more general setting. 
We will call the new method Generalized Operator based Prony Method (GOP). 
For that purpose, we start with recalling the generalized Prony method from \cite{Peter.2013b}. 

\subsection{Generalized Prony method}

Let $V$ be a normed vector space over ${\mathbb C}$
and let  $A: V \to V$  be a linear operator.
Assume that  $A$ possesses  a non-empty point spectrum $\sigma_{P}(A)$ and let $\sigma(A) \subset \sigma_{P}(A)$ be a (sub)set  with pairwise different 
eigenvalues of $A$. We assume further that there is a corresponding set of eigenfunctions, i.e., for each $\lambda \in \sigma(A) $
we have a $v_{\lambda} \in V$ with $A \, v_{\lambda} = \lambda v_{\lambda}$, and the mapping $\lambda \mapsto  v_{\lambda}$
is injective. In other words, the eigenspace to $\lambda$ is one-dimensional, or, if this is not the case, we have to determine  one relevant eigenfunction $v_{\lambda}$ corresponding to $\lambda$ in advance, which may occur in the expansion that we want to recover. Throughout the paper, we will assume that the considered eigenfunctions $v_{\lambda}$ are normalized, i.e., $\| v_{\lambda} \|_{V} = 1$.

We want to reconstruct  $M$-sparse expansions into eigenfunctions of $A$ of the form 
\begin{equation}\label{genexpo}
f = \sum_{j=1}^{M} c_{j} \, v_{\lambda_{j}} 
\end{equation}
where $\lambda_{j} \in  \sigma(A)$ and where we always assume $c_{j} \in {\mathbb C} \setminus \{ 0 \}$ for $j=1, \ldots , M$. 
The considered set of possible expansions is given as 
\begin{equation}\label{ma}
\mathcal{M}(A):=\left\{f=\sum\limits_{j=1}^{M}c_{j}\, v_{\lambda_{j}}:\; M<\infty, \, c_{j}\in\C\setminus \{ 0\}, \, 
\lambda_{j} \in \sigma(A), \lambda_{j} \neq \lambda_{k}\,  \textrm{for}\,  j \neq k  \right\}.
\end{equation}
The generalized Prony method in \cite{Peter.2013b} provides an algorithm to recover $f$ using only $2M$ complex measurements.
For that purpose, a linear functional $F: V \to {\mathbb C}$ is introduced that satisfies $F (v_{\lambda}) \neq 0$ for all $\lambda \in \sigma(A)$.

\begin{theorem}[Generalized Prony method \cite{Peter.2013b}]
\label{GPM}
With the assumptions above, the expansion $(\ref{genexpo})$  of eigenfunctions $v_{\lambda_{j}}$  of the linear operator $A$ can be uniquely reconstructed  from the values 
$F(A^{k}f)$, $k=0, \ldots , 2M-1$.
\end{theorem}

\begin{proof}
We give an outline of the proof in \cite{Peter.2013b} with our notation.
Observe that $f$ is completely  reconstructed if  we recover the subset $\Lambda_{f}:= \{ \lambda_{1}, \ldots , \lambda_{M} \} \subset \sigma(A)$ of ``active eigenvalues'' and the complex coefficients $c_{j}$, $j=1, \ldots , M$.
The eigenfunctions $v_{\lambda_{j}}$ are then uniquely determined by $\lambda_{j}$.

Let $P(z) = \prod_{j=1}^{M}(z- \lambda_{j}) = \sum_{\ell=0}^{M} p_{\ell} \, z^{\ell}$ 
 be the Prony polynomial determined by the set of $M$ pairwise different (unknown) active eigenvalues $\lambda_{j} \in \Lambda_{f}$, and 
${\mathbf p} = (p_{0}, \ldots , p_{M-1}, p_{M})^{T}$ with $p_{M}=1$ denotes the vector of its monomial coefficients. Then we obtain by (\ref{genexpo})
\begin{equation}\label{3.0} P(A f) = \prod_{k=1}^{M} (A - \lambda_{k} I) \, f = \sum_{j=1}^{M} c_{j} \prod_{k=1}^{M} (A - \lambda_{k} I) \, v_{\lambda_{j}} =0,
\end{equation}
and therefore
\begin{equation}\label{3.1} F(A^{k} \, P(A) \, f) = F \Big( A^{k} \Big( \sum_{\ell=0}^{M} p_{\ell} \, A^{\ell} f \Big) \Big) = \sum_{\ell=0}^{M} p_{\ell} \, F(A^{\ell+k} f) =0 
\end{equation}
for all $k \in {\mathbb N}$. Taking $M$ equations for $k=0, \ldots ,M-1$, is already sufficient  to recover the coefficient vector ${\mathbf p}$, since the matrix
$$ \left( F ( A^{\ell + k} \, f) \right)_{k=0,\ell=0}^{M-1,M}$$
has full rank $M$. This can be seen from the  factorization 
\begin{eqnarray} \nonumber
\left( F ( A^{\ell + k} \, f) \right)_{k=0,\ell=0}^{M-1,M} &=&
\Big( F ( A^{\ell + k} \sum_{j=1}^{M} c_{j} v_{\lambda_{j}}) \Big)_{k=0,\ell=0}^{M-1,M}
= \Big( \sum_{j=1}^{M} c_{j} F ( A^{\ell+k} v_{\lambda_{j}}) \Big)_{k=0,\ell=0}^{M-1,M} \\
\nonumber
&=& \Big( \sum_{j=1}^{M} c_{j} F( v_{\lambda_{j}}) \, \lambda_{j}^{\ell+k} \Big)_{k=0,\ell=0}^{M-1,M} \\
\label{2stern}
&=& V_{\Lambda_{f},M,M} \, \text{diag} \, (c_{j} \, F(v_{\lambda_{j}}))_{j=1}^{M} \, V_{\Lambda_{f},M+1,M}^{T} 
\end{eqnarray}
with the Vandermonde matrices 
$$ V_{\Lambda_{f},M,M} := (\lambda_{j}^{k})_{k=0,j=1}^{M-1,M}, \qquad V_{\Lambda_{f},M+1,M} := (\lambda_{j}^{k})_{k=0,j=1}^{M,M} $$
having full rank $M$.
Thus, we can first compute ${\mathbf p}$ as the right eigenvector of $(F(A^{\ell+k} f))_{k=0,\ell=0}^{M-1,M}$ to the eigenvalue $0$ with normalization $p_{M}=1$, determine $P(z)$, then  extract the zeros $\lambda$ of $P(z)$ to recover  $\lambda_{j}$, $j=1, \ldots ,M$,  and finally compute the coefficients $c_{j}$, $j=1, \ldots , M$, by solving  an overdetermined linear system of the form 
$$ F(A^{k} f) = \sum_{j=1}^{M} c_{j} \, \lambda_{j}^{k} \, F(v_{\lambda_{j}}), \qquad k=0, \ldots , 2M-1. $$
\end{proof}
\medskip

\begin{remark}
As shown in \cite{Peter.2013b} and \cite{PSK18}, many expansions fit into the scheme of Theorem \ref{GPM}. In Section \ref{sec_GPM} we have used $A$ to be the shift operator or the differential operator. Other examples in \cite{Peter.2013b} and \cite{PSK18} include the dilation operator,  generalized shift operators as well as the Sturm-Liouville differential operator of second order.
\end{remark}

\subsection{Generalization 1: Change of operators}

The actions $A^{k} f$ needed for the generalized Prony method  to recover $f \in {\mathcal M}(A)$ in (\ref{ma}) may be very expensive to acquire. Therefore we can  try to replace the operator $A$ by a different operator with the same eigenfunctions $v_{\lambda}$ such that the powers of this new operator are simpler to realize.
 We start with the following definition.

\begin{definition}[Iteration Operator]
\label{defevo} Let $A:V\to V$ be a
linear operator, and let $\sigma(A) \neq \emptyset$  be a subset of the  point spectrum $\sigma_P(A)$ with pairwise different eigenvalues and
with corresponding normalized eigenfunctions
$v_{\lambda}$
such that the map $\lambda \mapsto v_{\lambda}$ is injective for $\lambda \in \sigma(A)$.
Further, let $\varphi: \sigma(A) \to \C$
be an injective function. We call $\Phi = \Phi_{\varphi}$ an \textit{iteration operator to $A$} if $\Phi:  {\mathcal M}(A) \to {\mathcal M}(A)$ is a well-defined linear operator and $\Phi \, v_{\lambda} = \varphi(\lambda) \, v_{\lambda}$
for all $\lambda \in \sigma(A)$. 
\end{definition}

The injectivity of $\varphi$ in Definition \ref{defevo} implies that the values $\varphi(\lambda)$ are  pairwise different for all $\lambda \in \sigma(A)$. In particular, we can show that for analytic functions $\varphi$ the operator $\Phi= \varphi(A)$ is an iteration operator.

\begin{theorem}
\label{th_spectral_mapping}
Let $A:V\to V$ be a
linear operator, and let $\sigma(A) \neq \emptyset$  be a subset of the  point spectrum $\sigma_P(A)$ with pairwise different eigenvalues and
with corresponding eigenfunctions
$v_{\lambda}$
such that the map $\lambda \mapsto v_{\lambda}$ is injective for $\lambda \in \sigma(A)$. Let $\varphi: \sigma(A) \to {\mathbb C}$ be an analytic,  injective function.
Then $\varphi(A)$ is an iteration operator, i.e., it is a well-defined  linear operator on ${\mathcal M}(A)$ and 
$$(A - I\lambda)\, v_{\lambda}= 0 $$
implies
$$
(\varphi(A) - \varphi(\lambda)I) \, v_{\lambda} = 0.$$
This means, if $v_{\lambda}$ is an eigenfunction of $A$ corresponding to the eigenvalue $\lambda$, then $v_{\lambda}$ is also an eigenfunction of $\varphi(A)$ corresponding to the eigenvalue $\varphi(\lambda)$. 
\end{theorem}

\begin{proof}
Since $\varphi$ is assumed to be analytic on $\sigma(A)$, it follows that its power series $\varphi(z) = \sum_{n=0}^{\infty} a_{n} \, z^{n}$ converges for $z \in \sigma(A)$. Thus, $A \, v_{\lambda} = \lambda \, v_{\lambda} $ implies for all $\lambda \in \sigma(A)$
$$ \varphi(A) \, v_{\lambda} = \sum_{n=0}^{\infty} a_{n} \, A^{n} v_{\lambda} = \lim_{N\to \infty} \sum_{n=0}^{N}  a_{n} \lambda^{n} v_{\lambda} = \varphi(\lambda) \, v_{\lambda}.$$
Further, the injectivity of $\varphi$ implies that the eigenvalues $\varphi(\lambda)$, $\lambda \in \sigma(A)$, are pairwise distinct.
Thus, $\varphi(A)$ is well-defined on ${\mathcal M}(A)$ and satisfies all assumptions of an iteration operator.
\end{proof}

\begin{example}\label{exs}
1. One example has been  already seen in Section \ref{sec_GPM}. We can take $V=C^{\infty}({\mathbb R})$, $A= \frac{\mathrm d}{{\mathrm d}x}$ with $\sigma_{P}(A) = {\mathbb C}$ according to Theorem \ref{th_spectral_mappingex}. Further, let $\sigma(A) = {\mathbb R} + {\mathrm i} \, [-C \pi, \, C \pi) \subset \sigma_{P}(A)$.
Then, $\varphi(z):=\exp(\tau \, z)$ with $0 < \tau \le 1/C$ is injective on $\sigma(A)$,
and we obtain the iteration operator $\varphi(A) = S_{\tau}$ on ${\mathcal M}(A)$.
\smallskip

2. We take $\varphi(z) = z^{-1}$ and $\sigma(A) \in \sigma_{P}(A) \setminus \{ 0 \}$. Then $\varphi(A) = A^{-1}$ is well-defined on ${\mathcal M}(A)$  and
$$A \, v_{\lambda}=\lambda \, v_{\lambda} \qquad  \Leftrightarrow \qquad 
A^{-1}\, v_{\lambda}=\frac{1}{\lambda} \,  v_{\lambda}.$$
For example,  $A=S_{\tau}$ with $\tau \neq 0$ yields $A^{-1}=S_{-\tau}$. The dilation operator $D_{a}:C({\mathbb R}) \to C({\mathbb R})$ 
with $D_{a}f (x) := f(a x)$, $a \neq 0$ and $|a| \neq 1$, yields $D_{a}^{-1} f (x) = f(\frac{1}{a} x)$.

3. Consider the operator $A $ on $C^{\infty}({\mathbb R})$ given by 
$$ Af (x) :=  x \, \frac{{\mathrm d} f }{{\mathrm d}x}(x) = x \, f'(x)$$
with eigenfunctions $x^{p}$ for $p \in {\mathbb R}$ to the eigenvalues $p \in {\mathbb R}$.
We use  $\varphi(z) = \exp( \tau \, z)$ with $\tau \in {\mathbb R} \setminus \{ 0 \}$ and obtain for each polynomial $x^{m}$ that 
$$ \exp( \tau x \, \frac{\mathrm d}{{\mathrm d}x}) \, x^{m} = \sum_{\ell=0}^{\infty} \frac{\tau^{\ell}}{\ell !} \, \left( x \, \frac{\mathrm d}{{\mathrm d}x} \right)^{\ell} \, x^{m} = \sum_{\ell=0}^{\infty} \frac{\tau^{\ell}}{\ell !}  \, m^{\ell} \, x^{m} ={\mathrm e}^{\tau m} \, x^{m}  = ({\mathrm e}^{\tau} x)^{m},
$$
see also \cite{Dattoli.2000}.
Thus, $\varphi(A)$ is here the dilation operator $D_{\exp(\tau)}$. The injectivity condition for $\varphi(z)$ is satisfied  since $\exp(\tau p)$ is strictly monotone as a function in $p$. $\Box$
\end{example}
\medskip

What does a change from $A$ to $\varphi(A)$ mean for the reconstruction scheme to recover an expansion $f$ in (\ref{genexpo})?
Using the operator $A$ and a functional $F$, Theorem \ref{GPM} implies that we need (at least) the sample values $F(A^{k} f)$, $k=0, \ldots , 2M-1$ for the recovery of $f$. Changing from $A$ to $\varphi(A)$, we observe that all assumptions required in Theorem \ref{GPM} also hold for $\varphi(A)$, and we can now reconstruct $f$ in (\ref{genexpo}) from samples $F(\varphi(A)^{k} f)$, $k=0, \ldots , 2M-1$, thereby employing the new Prony polynomial 
$$P_{\varphi}(z):=\prod_{j=1}^{M}\left(z-\varphi(\lambda_{j}
)  \right):=\sum\limits_{\ell=0}^M p_{\ell}\, z^{\ell}. $$
Taking a suitable $\varphi$ may have two advantages. First, the samples $F(\varphi(A)^{k} f)$, $k=0, \ldots , 2M-1$, may be much simpler to acquire. In Section \ref{sec:cos} and Section \ref{sec_GEProM}, we will present many examples, where a change from linear differential operators $A$ to generalized shift operators $\varphi(A)$ leads to new recovery schemes for the expansions in (\ref{genexpo}) employing  just function values of $f$ instead of high order derivative values.

Second, the numerical scheme to recover $f$ can be essentially stabilized. The main reason for that is the change of eigenvalues from $\lambda \in \Lambda_{f}$ to $\varphi(\lambda) \in \varphi(\Lambda_{f})$. The eigenvalues play an important role  for the matrices being involved in the Prony algorithms. Compared with the generalized Prony method, we get now  instead of (\ref{2stern}) the Hankel matrix  factorization
$$ \left( F ( \varphi(A)^{\ell + k} \, f) \right)_{k=0,\ell=0}^{M-1,M} = V_{\varphi(\Lambda_{f}),M,M} \, \text{diag} \, (c_{j} \, F(v_{\lambda_{j}}))_{j=1}^{M} \, V_{\varphi(\Lambda_{f}),M+1,M}^{T} $$
 with the Vandermonde matrices 
$$ V_{\varphi(\Lambda_{f}),M,M} := (\varphi(\lambda_{j})^{k})_{k=0,j=1}^{M-1,M}, \qquad V_{\varphi(\Lambda_{f}),M+1,M} := (\varphi(\lambda_{j})^{k})_{k=0,j=1}^{M,M} $$
to recover the coefficient vector ${\mathbf p} =(p_{0}, \ldots , p_{M})^{T}$ of the Prony polynomial $P_{\varphi}$.

\subsection{Generalization 2: Change  the sampling scheme}
\label{sec3.3}

As we have seen  in Theorem \ref{GPM} and Theorem \ref{th_spectral_mapping},  the expansion $f= \sum_{j=1}^{M} c_{j} \,  v_{\lambda_{j}}$ into eigenfunctions of the operator $A$ can be recovered using either the samples $F(A^{k} f)$ or 
the samples  $F(\varphi(A)^{k} f)$ for  $k=0, \ldots , 2M-1$, where $F:V \to {\mathbb C}$ is a linear functional satisfying $ F(v_{\lambda}) \neq 0 $ for all $\lambda \in \sigma(A)$.
Having a closer look at the equations (\ref{3.0}) and (\ref{3.1}) we observe however that  already $P(A) f=0$, such that $F \, A^{k}$ can be replaced by different functionals.

\begin{definition}[Sampling  Functionals]
\label{def_evaluation_scheme}
\textit{Let $A:V \to V$ be a linear operator and let $\sigma(A)$ be a fixed  subset of pairwise different eigenvalues  of $A$.
Further, let 
$$V_{\sigma(A)}:= \{ v_{\lambda} : \, A \, v_{\lambda} = \lambda \, v_{\lambda}, \, \lambda \in \sigma(A), \, \| v_{\lambda} \|_{V} = 1 \}$$
be the corresponding set of eigenfunctions such that the mapping $\lambda \to v_{\lambda}$ is injective on $\sigma(A)$. 
Then $\{ F_{k} \}_{k=0}^{M-1}$ with 
$$F_k: \, V\to\C, \quad \quad k =0, \ldots , M-1, $$
forms an admissible set of sampling functionals for $A$ 
if  for all finite subsets $\Lambda_{M} \subset \sigma(A)$ with cardinality $M < \infty$ the matrix
$$ \left(F_{k} (v_{\lambda}) \right)_{k=0,\lambda \in \Lambda_{M}}^{M-1}$$
has full rank $M$.}
\end{definition}

If the set of functionals $\{ F_{k} \}_{k=0}^{M-1}$ is admissible for a linear operator $A$, then it is also admissible for any iteration operator $\varphi(A)$, since the eigenvectors  $v_{\lambda}$ do not change.
Then we obtain

\begin{theorem}\label{theosamp}
Assume that $\{F_{k} \}_{k=0}^{M-1}$  forms an admissible set of sampling functionals  for the linear operator $A: V \to V$ according to Definition $\ref{def_evaluation_scheme}$.
Let $f \in {\mathcal M}(A)$ be a linear expansion into eigenfunctions of $A$ as in $(\ref{genexpo})$. 
Then the sampling matrix
$$ \left(F_{k} (A^{\ell} f) \right)_{k=0,\ell=0}^{M-1,M} \in {\mathbb C}^{M \times (M+1)}$$
possesses rank $M$ and is called admissible sampling matrix for $f$.
Further, if $\Phi=\varphi(A)$ is an iteration operator of $A$ as given in Theorem \ref{th_spectral_mapping}, then also
$$ \left(F_{k} (\Phi^{\ell} f) \right)_{k=0,\ell=0}^{M-1,M} \in {\mathbb C}^{M \times (M+1)}$$
possesses rank $M$ and is therefore an admissible sampling matrix.
\end{theorem}

\begin{proof} We show the second equation for $\Phi=\varphi(A)$, where $\varphi$ is an injective analytic function on $\sigma(A)$. Then the first equation follows by taking $\varphi(z) = z$. We find
\begin{eqnarray*}
\Big( F_{k}(\varphi(A)^{\ell} f) \Big)_{k=0,\ell=0}^{M-1,M}\!\!\!\! &=& \!\!\!\!\! \Big(\! F_{k}( \varphi(A)^{\ell} \sum_{j=1}^{M} c_{j} \, v_{\lambda_{j}} ) \Big)_{k=0,\ell=0}^{M-1,M}
\!\!\!\!= \!\Big( \! \sum_{j=1}^{M} c_{j} \, \varphi(\lambda_{j})^{\ell} \,  F_{k} (v_{\lambda_{j}})  \Big)_{k=0,\ell=0}^{M-1,M} \\
&=& \Big( F_{k} (v_{\lambda_{j}}) \Big)_{k=0, j=1}^{M-1,M} \, 
\text{diag} \, (c_{j})_{j=1}^{M}\, \Big( \varphi(\lambda_{j})^{\ell} \Big)_{j=1, \ell=0}^{M,M}.
\end{eqnarray*}
All three matrices in this factorization  have full rank $M$ by assumption, and the assertion follows. In particular, the last matrix is a Vandermonde matrix generated by $M$ pairwise distinct values $\varphi(\lambda_{j})$,  $j=1, \ldots , M$.
\end{proof}

\begin{example}
Comparison  with formula (\ref{3.1}) yields that  $F_{k}= F A^{k}$, $k=0, \ldots , M-1$, is always an admissible set of sampling functionals,  since the proof of Theorem \ref{GPM} shows that $(F(A^{k+\ell} f))_{k=0,\ell=0}^{M-1, M}$ has full rank $M$  for each $f$ in ${\mathcal M}(A)$. $\Box$
\end{example}

Further we have

\begin{lemma}\label{lem:1}
Let $A:V\to V$ be a
linear operator, and let $\sigma(A) \neq \emptyset$  be a subset of the  point spectrum $\sigma_P(A)$ with pairwise different eigenvalues and
with corresponding eigenfunctions
$v_{\lambda}$
such that the map $\lambda \mapsto v_{\lambda}$ is injective for $\lambda \in \sigma(A)$.
Let $\psi$ be an analytic injective function  on $\sigma(A)$. Assume that  $F:V \to {\mathbb C}$ is a linear functional with $F v_{\lambda} \neq 0$ for all $\lambda \in \sigma(A)$.
Then $\{F_{k} \}_{k=0}^{M-1} := \{ F (\psi(A)^{k})\}_{k=0}^{M-1}$ is an admissible set of sampling functionals and the matrix 
$$ ( F_{k}(A^{\ell} f))_{k=0,\ell=0}^{M-1,M} = (F(\psi(A)^{k}\,  A^{\ell} f))_{k=0,\ell=0}^{M-1, M} \in {\mathbb C}^{M \times M+1} $$
is an admissible sampling matrix for each $f \in {\mathcal M}(A)$.
\end{lemma}

\begin{proof}
From  $\psi(A)^{k} \, v_{\lambda} = \psi(\lambda)^{k} \, v_{\lambda}$ it follows that 
$$F_{k}  (v_{\lambda}) = F (\psi(A)^{k} \, v_{\lambda}) = \psi(\lambda)^{k} \, F(v_{\lambda})$$ 
is bounded  and nonzero by assumption.
Further, for $f \in {\mathcal M}(A)$,
\begin{eqnarray*}
\Big(F(\psi(A)^{k}\,  A^{\ell} f) \Big)_{k=0,\ell=0}^{M-1, M} &=& 
\Big(F \Big( \psi(A)^{k}\,  A^{\ell} \, \sum_{j=1}^{M} 
c_{j} \, v_{\lambda_{j}} \Big) \Big)_{k=0,\ell=0}^{M-1, M} \\
&=& \Big( F \Big(  \sum_{j=1}^{M} c_{j} \, \psi(\lambda_{j})^{k} \,  
 \lambda_{j}^{\ell} \, v_{\lambda_{j}} \Big) \Big)_{k=0,\ell=0}^{M-1, M} \\
&=& {\mathbf V}_{\psi(\Lambda_{f}), M,M} \, \textrm{diag} \, ((c_{j} \, 
F(v_{\lambda_{j}})))_{j=1}^{M} \, {\mathbf V}_{\Lambda_{f},M+1,M}^{T}
\end{eqnarray*}
with , with 
$\Lambda_{f} = \{\lambda_{1}, \ldots , \lambda_{M}\}$, ${\mathbf V}_{\psi(\Lambda_{f}), M,M} := ((\psi(\lambda_{j}))^{k})_{k=0,j=1}^{M-1,M}$ and ${\mathbf V}_{\Lambda_{f},M+1,M} := (\lambda_{j}^{\ell})_{\ell=0,j=1}^{M,M}$.
These two  Vandermonde matrices have full rank $M$ since the $\lambda_{j} \in \Lambda_{f}$ are pairwise different and $\psi$ is injective on $\Lambda_{f}$ with $\psi(\lambda_{j}) \neq 0$ for $\lambda_{j} \in \Lambda_{f}$. 
\end{proof}
\medskip

\subsection{Generalized operator based Prony method (GOP)}
\label{sec3.neu}

The following theorem summarizes the central statement of the generalized operator-based Prony method (GOP) 
and the corresponding proof  results  in an algorithm to solve the reconstruction problem for $f\in\mathcal{M}(A)$ in (\ref{ma}).

\begin{theorem}[Generalized Operator based Prony Method] \label{GOProM} \hspace{1cm} \\
Let  $A: V \to V$ be a linear operator on the normed vector space $V$ over ${\mathbb C}$, and let $\sigma(A)$ be a subset of pairwise different eigenvalues of $A$. Let $\Phi=\varphi(A)$ be  an iteration operator of $A$ as given in Definition $\ref{defevo}$.
Assume that  the set $\{ F_{k} \}_{k=0}^{M-1}$ is an admissible set of sampling functionals according to Definition $\ref{def_evaluation_scheme}$.
 Then each $f \in {\mathcal M}(A)$ can be completely recovered from the complex samples $F_{k}(\varphi(A)^{\ell} f)$, $k=0, \ldots , M-1$, $\ell=0, \ldots , M$.
\end{theorem}

\begin{proof}
To recover $f = \sum_{j=1}^{M} c_{j} \, v_{\lambda_{j}} \in {\mathcal M}(A)$, we only have to determine the set $\Lambda_{f}= \{ \lambda_{1}, \ldots , \lambda_{M} \}$ of  ``active eigenvalues'' and the corresponding coefficients $c_{j} \in {\mathbb C} \setminus  \{ 0 \}$, $j=1, \ldots , M$, since the map $\lambda \to v_{\lambda}$ is assumed to be injective. 
Further, since $\varphi$ is also injective on $\sigma(A)$, we can determine the set $\varphi(\Lambda_{f}) = \{\varphi(\lambda_{j}): \, j=1, \ldots , M \}$ instead of $\Lambda_{f}$ by Theorem \ref{th_spectral_mapping}. 

Let now 
$$P_{\varphi}(z):= \prod_{j=1}^{M} (z - \varphi(\lambda_{j})) = \sum_{\ell=0}^{M} p_{\ell} \, z^{\ell} $$
be the Prony polynomial determined be the unknown pairwise different active eigenvalues $\varphi(\lambda_{j})$ of $\varphi(A)$ for $\lambda_{j} \in \Lambda_{f}$,
where ${\mathbf p} = (p_{0}, \ldots , p_{M-1}, p_{M})^{T}$ with $p_{M} = 1$ denotes the vector  of coefficients in the monomial representation  of $P_{\varphi}(z)$.  Then
\begin{eqnarray*}
P_{\varphi}(\varphi(A)) f &=&  \prod_{k=1}^{M} ( \varphi(A) - \varphi(\lambda_{k}) I) \, f \\
&=& \sum_{j=1}^{M} c_{j} \, \prod_{k=1}^{M} (\varphi(A) - \lambda_{k} I) v_{\lambda_{j}} =0, 
\end{eqnarray*}
and therefore
$$ F_{k} (P_{\varphi}(\varphi(A)) f) = F_{k} \Big( \sum_{\ell =0}^{M} p_{\ell} \, \varphi(A)^{\ell} f \Big) = \sum_{\ell =0}^{M} p_{\ell} F_{k}(\varphi(A)^{\ell} \, f) = 0, \; k=0, \ldots , M-1.$$
Thus, we obtain a homogeneous linear system to compute ${\mathbf p}$, where  by Theorem \ref{theosamp} (with $A$ replaced by $\varphi(A)$) the coefficient matrix  is the admissible sampling matrix 
$( F_{k}(\varphi(A)^{\ell} f))_{k=0,\ell=0}^{M-1,M} \in {\mathbb C}^{M \times M+1}$ with full rank $M$.
Hence, ${\mathbf p}$ is uniquely determined by this system using the normalization $p_{M}=1$. We can now extract the zeros $\varphi(\lambda_{j})$, $j=1, \ldots, M$, and thus $\Lambda_{f}= \{\lambda_{1}, \ldots, \lambda_{M}\}$.
Finally, we compute the coefficients $c_{\lambda}$ as solutions of the linear system 
\begin{equation}
\label{sys1} F_{k}(\varphi(A)^{\ell} f) = \sum_{j=1}^{M} c_{j} \varphi(\lambda_{j})^{\ell} F_{k} (v_{\lambda_{j}}), \quad \ell=0, \ldots , M,
\end{equation}
where the coefficient matrix is of full rank, since $F_{k} (v_{\lambda_{j}}) \neq 0$ and the arising Vandermonde matrix $((\varphi(\lambda_{j}))^{\ell})_{\ell=0, j=1}^{M,M}$ has full rank $M$ since the values $\varphi(\lambda_{j})$, $j=1, \ldots , M$, are pairwise different.
\end{proof}
\medskip

The proof of Theorem \ref{GOProM} is constructive and leads to the following algorithm for the 
recovery of $f \in {\mathcal M}(A)$.
We assume here that we have an iteration operator $\varphi(A)$ and a given set of admissible sampling functionals $F_{k}$ such that the sampling matrix  $( F_{k}(\varphi(A)^{\ell} f))_{k=0,\ell=0}^{M-1,M} \in {\mathbb C}^{M \times M+1}$ for the operator $\varphi(A)$ has full rank $M$.

\begin{algorithm}[GOP] \hspace{1cm} \\
{\bf Input}: $F_k\left(\varphi(A)^{\ell}f \right)$, $\ell =0, \ldots , M$, $k=0, \ldots , M-1$, where $f \in {\mathcal M}(A)$.
\begin{itemize}
	\item  Compute the kernel vector ${\mathbf p}=(p_{0}, \ldots , p_{M-1}, p_{M})^{T}$ with $p_{M}=1$ of the   matrix  
	$( F_{k}(\varphi(A)^{\ell} f))_{k=0,\ell=0}^{M-1,M} \in {\mathbb C}^{M \times M+1}$.
	\item Compute the $M$ zeros $\varphi(\lambda_j)$, $j=1, \ldots , M$, of the Prony polynomial $P_{\varphi}(z) = \sum_{\ell=0}^{M}
	p_{\ell} z^{\ell}$ and identify the active eigenfunctions $v_{\lambda_{j}}$ by $\varphi(A) \, v_{\lambda_{j}} = \varphi(\lambda_{j}) \, v_{\lambda_{j}}$. 
	Compute $\lambda_{j}$ from $\varphi(\lambda_{j})$ to obtain $\Lambda_{f} = \{ \lambda_{1}, \ldots , \lambda_{M} \}$. 
	\item Compute $c_{j}$ by solving the system in (\ref{sys1}).
\end{itemize}	
{\bf Output}: Parameters  $\lambda_{j}$ and  $c_{j}$, $j=1, \ldots , M$ such that $f = \sum\limits_{j=1}^{M} c_{j} \, v_{\lambda_{j}}$.
\end{algorithm}

\begin{remark}
1. The generalized Prony method in \cite{Peter.2013b} is a special case of GOP if we take $\varphi(z) = z$ and $F_{k} = F \, A^{k}$ for some suitable functional $F$. In this case the sampling  matrix has Hankel structure and we need only $2M$ input values.

2. If we choose $F_{k} = F (\psi(A)^{k} \cdot )$ for some analytic function $\psi$ as in Lemma \ref{lem:1}, then the sampling matrix can be taken in the form $ (F (\psi(A)^{k} \, \varphi(A)^{\ell}f ))_{k=0,\ell=0}^{M-1,M} \in {\mathbb C}^{M \times M+1}$, where compared to Lemma \ref{lem:1}, we have replaced the powers of $A$ by powers of $\varphi(A)$. This sampling matrix is also admissible, and the proof can be performed as for Lemma \ref{lem:1}.

3. GOP can be also generalized to operators with eigenvalues of higher geometric multiplicity, similarly as the generalized Prony method, \cite{Peter.2013b}.
This approach leads to a Prony polynomial with  zeros of higher multiplicity.
We also refer to \cite{BY13, Mourrain.2017}.
In this paper we restrict ourselves to the case where the correspondence between $\lambda$ resp. $\varphi(\lambda)$ and $v_{\lambda}$ is bijective. 
\end{remark}

\subsection{Application of GOP to cosine expansions}
\label{sec:cos}
In this section, we want  to explain the ideas of GOP in a simple example. 

Consider the expansion 
\begin{equation}\label{cos1}
f(x) :=\sum\limits_{j=1}^M c_j \, \cos(\alpha_{j} x),
\end{equation} 
where we want to recover the $2 M $ parameters $\alpha_{j} \in [0, \, C) \subset {\mathbb R}$ and $c_{j} \in {\mathbb C} \setminus \{0\}$, $j=1, \ldots , M$.
We observe that  $A:=-\frac{{\mathrm d}^2}{{\mathrm d}x^2}$ is an operator on $C^{\infty}({\mathbb R})$ such that all functions $\cos(\alpha  x)$ are eigenfunctions of $A$ with 
$$A \cos (\alpha \cdot) = \alpha^{2} \cos(\alpha \cdot). $$ 
Using the generalized Prony method in Theorem \ref{GPM}, we can therefore 
reconstruct $f$ in (\ref{cos1})  using the samples $F(A^{k}f) = (-1)^{k} \, F(f^{(2k)}) $,  
$k=0, \ldots, 2M-1$, where $f^{(2k)}$   denotes the $2k$-th derivative of 
$f$. Here, the  sampling functional  $F: C^{\infty}({\mathbb R}) \to {\mathbb C}$ 
needs to satisfy $F(\cos(\alpha \cdot) ) \neq 0$ for all all $\alpha \in [0, C)$.

 Taking e.g.\ the point evaluation functional $F f  = f(0)$, we need the measurements $f^{(2k)}(0)$, $k=0, \ldots , 2M-1$. These measurements are usually difficult to provide, it would be much better to use just function values of $f$.
 
\medskip

We want to apply now GOP in Theorem \ref{GOProM} to reconstruct $f$ in (\ref{cos1}) in a different way.
We employ the analytic function $\varphi(z)$ of the form
$$ \varphi(z) = \sum_{n=0}^{\infty} (-1)^{n} \frac{\tau^{2n} \, z^{n}}{(2n)!},$$
i.e., $\varphi(z^{2}) = \cos(\tau z)$, and observe that the application of $\varphi(A)$ to monomial functions $x^{m}$ gives
\begin{eqnarray*}
\varphi(A) \, x^{m} &=& \sum_{n=0}^{\infty} (-1)^{n} \frac{\tau^{2n}}{(2n)!} \left( - \frac{{\mathrm d}^2}{{\mathrm d}x^2} \right)^{n} \, x^{m} \\
&=& \sum_{0 \le 2n \le m}  \binom{m}{2n} 
\tau^{2n} \, x^{m-2n} \\
&=& \frac{1}{2} \left( \sum_{0 \le n' \le m} \binom{m}{n'}  \tau^{n'} \, x^{m-n'} + \sum_{0 \le n' \le m}   \binom{m}{n'} (-\tau)^{n'} \, x^{m-n'} \right) \\
&=& \frac{1}{2} \left( (x+\tau)^{m} + (x-\tau)^{m} \right) 
= \frac{1}{2} ( S_{\tau} + S_{-\tau}) \, x^{m}
\end{eqnarray*}
with the shift operator $S_{\tau}$ given by $S_{\tau} f = f(\cdot + \tau)$.
Thus we have
$$ \varphi(A) = \frac{1}{2} ( S_{\tau} + S_{-\tau}) $$
and by Theorem \ref{th_spectral_mapping} it follows that
\begin{eqnarray*} \varphi(A) \, \cos(\alpha \cdot) &=& \frac{1}{2} ( S_{\tau} + S_{-\tau}) \, \cos(\alpha \cdot) = \frac{1}{2} ( \cos(\alpha(\cdot + \tau)) + \cos(\alpha(\cdot - \tau))) \\
&=&  \cos( \alpha \tau) \, \cos(\alpha \cdot), 
\end{eqnarray*}
i.e., the eigenvalues $\alpha^{2}$ of $A= - \frac{{\mathrm d}^{2}}{{\mathrm d} x^{2}}$ are transferred to $\cos(\tau \alpha)$. 
We can still identify $\alpha \in [0, C)$ uniquely from $\cos(\tau \alpha)$ if $\tau \le \frac{\pi}{C}$.

In order to apply GOP, we also need to fix  an admissible sampling matrix.
According to Lemma \ref{lem:1}, we can use an admissible set of sampling functionals 
\begin{equation}\label{sampi0} F_{k} = F (\varphi(A)^{k} ) = F \, ( \frac{1}{2} \left( S_{\tau} + S_{-\tau}) \right)^{k} = F \, \left( \frac{1}{2^{k}} \sum_{r=0}^{k} \binom{k}{r} \, S_{(k-2r)\tau} \right)
\end{equation}
and arrive with the point evaluation functional $F f: = f(0) $ at the sampling matrix $\left( F_{k} ( \varphi(A)^{\ell} f)\right)_{k=0, \ell=0}^{M-1,M}$ with entries
$$ F_{k} ( \varphi(A)^{\ell} f ) =   F (\varphi(A)^{k+\ell} f )
=    \frac{1}{2^{k+\ell}} \sum_{r=0}^{k+\ell} \binom{k+\ell}{r} \, f( (k+\ell-2r)\tau)  .
$$
This matrix involves the function samples $f(k\tau)$, $-2M+1 \le k \le 2M-1$. Since $f$ in (\ref{cos1}) is symmetric, it is sufficient to provide $f(k\tau)$, $k=0, \ldots , 2M-1$. Indeed,
\begin{eqnarray*}   F (\varphi(A)^{k+\ell} f)
&=&  \sum_{j=1}^{M} c_{j} \, F (\varphi(A)^{k+\ell}  \cos(\alpha_{j} \cdot)) \\
&=&  \sum_{j=1}^{M} c_{j} \, (\cos(\alpha_{j} \tau ))^{\ell+k} \, F (\cos(\alpha_{j} \cdot)) =  \sum_{j=1}^{M} c_{j} \, (\cos(\alpha_{j} \tau ))^{\ell+k}
\end{eqnarray*}
yields that the sampling matrix can be simply factorized, and all matrix factors have full rank $M$.
\smallskip

We can employ a different sampling matrix by taking 
$$F_{k} (f) = ((S_{k\tau}+ S_{-k\tau}) f)(0)$$ 
instead of (\ref{sampi0}) and get the matrix entries
\begin{equation}\label{sampi1}
(( S_{k\tau} + S_{-k\tau}) (\varphi(A)^{\ell} f))(0) = \frac{1}{2^{\ell}} \sum_{r=0}^{\ell} \binom{\ell}{r} [ f((\ell + k -2r) \tau) + f(\ell - k- 2r) \tau ) ]. 
\end{equation}
For $f$ of the form (\ref{cos1})  this sampling matrix is also admissible  since we obtain with the Chebyshev polynomial $T_{k}(z) := \cos ( k (\arccos z))$ that 
\begin{eqnarray*}
& & \hspace*{-20mm} (( S_{k\tau} + S_{-k\tau}) \varphi(A)^{\ell} f)(0) \\
&=& \frac{1}{2^{\ell}} \sum_{r=0}^{\ell} \binom{\ell}{r} \sum_{j=1}^{M} c_{j}[ \cos(\alpha_{j}(\ell + k -2r) \tau) + \cos(\alpha_{j}(\ell - k- 2r) \tau ) ] \\
&=& \sum_{j=1}^{M} c_{j} \left( \frac{2}{2^{\ell}}  \sum_{r=0}^{\ell} \binom{\ell}{r} \cos(\alpha_{j}(\ell-2r)\tau) \right) \, \cos(\alpha_{j} k \tau) \\
&=& \sum_{j=1}^{M} c_{j} \left( \frac{2}{2^{\ell}}  \sum_{r=0}^{\ell} \binom{\ell}{r} T_{|\ell-2r|}(\cos(\alpha_{j}\tau)) \right) \, \cos(\alpha_{j} k \tau) \\
&=& 2 \sum_{j=1}^{M} c_{j} \, (\cos(\alpha_{j} \tau))^{\ell} \, \cos(\alpha_{j} k \tau),
\end{eqnarray*} 
where we have used the identity $x^{\ell} = \frac{1}{2^{\ell}}\sum_{r=0}^{\ell} \binom{\ell}{r} T_{|\ell-2r|}(x)$. Thus
{\small 
$$ \left( ((S_{k\tau} + S_{-k\tau}) \varphi(A)^{\ell} f)(0)\right)_{k=0,\ell=0}^{M-1,M} = (\cos(\alpha_{j} k \tau))_{k=0,j=1}^{M-1,M} \, \textrm{diag} \, (2c_{j})_{j=1}^{M} \, ((\cos(\alpha_{j} \tau))^{\ell})_{j=1,\ell=0}^{M,M},
$$
}
where all matrix factors have full rank $M$.
The sampling matrix in (\ref{sampi1}) applies the idea that instead of 
$F_{k}(f) = F( \varphi(A)^{k} f)$, $k=0, \ldots, M-1$, we can also use 
$$ F_{k}(f) = F(p_{k}(\varphi(A) f)), \qquad k=0, \ldots , M-1,$$
with a basis $\{p_{k}\}_{k=0}^{M-1}$ of the space of algebraic polynomials up to degree $M-1$. Here, (\ref{sampi1}) is obtained by using the basis of Chebyshev polynomials $p_{k}=T_{k}$, $k=0, \ldots , M-1$. 

\begin{remark}
A slightly different sampling scheme was applied in \cite{Potts.2014} and in \cite{PSK18}, where the Prony polynomial has been written using a  Chebyshev polynomial basis instead of the monomial basis.
\end{remark}

\section{GOP for special linear differential operators of first and second order} 
\label{sec_GEProM}
\setcounter{equation}{0}

In this section we discuss the application of GOP for the recovery of expansions into eigenfunctions of linear differential operators.
In this case, we will mainly apply iteration operators that are constructed using $\varphi(z) = \exp(\tau z)$ and $\varphi(z) = \cos(\tau z^{1/2})$.
We will show that the obtained iteration operators are generalized shift operators that enable us to recover the considered expansions using only function values instead of derivative values.
We will consider sampling functionals $F_{k}: {\mathcal M} \to {\mathbb C}$ of the form 
$$ F_{k} (f)= F(\varphi(A)^{k} f). $$
With this sampling, GOP is equivalent with the  generalized Prony method  for $\varphi(A)$ (instead of $A$)  and a fixed functional $F$ that only needs to satisfy the assumptions of Theorem \ref{GPM}.
Then,  the corresponding sampling matrix is always admissible for all $f \in {\mathcal M}(A)$ in (\ref{ma}), and we need the values
$F((\varphi(A)^{k} f)$, $k=0, \ldots , 2M-1$ to reconstruct $f$ in (\ref{genexpo}).

\subsection{Differential operators of first order and generalized shifts}
\label{sec_Generalized_Shifts}

Assume that  $G: I \to J \subset {\mathbb R}$ is in $C^{\infty}(I)$ and that its first derivative $G'(x)$ has no zero on $I$.  This means in particular that $g(x) = 1/G'(x)$ is well-defined  on $I$.
Moreover, $G(x)$ is strictly monotone on $I$ such that $G^{-1}(x)$ is also well-defined on $I$. Further, let $H \in C^{\infty}(I)$.
\smallskip

We want to reconstruct  functions of the form 
\begin{equation}\label{genex}
f(x) = \sum_{j=1}^{M} c_{j} \, {\mathrm e}^{H(x) + \lambda_{j} G(x)}, 
\end{equation}
i.e.,  we want to recover the parameters $c_{j} \in {\mathbb C} \setminus \{0 \}$ and $\lambda_{j} \in {\mathbb R} + {\mathrm i} [-C,C)$.
We define the functions 
\begin{equation}\label{gh} g(x) := \frac{1}{G'(x)}, \qquad h(x) := - \frac{H'(x)}{G'(x)}. 
\end{equation}
Then $v_{\lambda_{j}}(x) := {\mathrm e}^{H(x) +\lambda_{j} G(x)}$  are eigenfunctions of 
\begin{equation}\label{4.1.1}
A = g(\cdot) \frac{ {\mathrm d} }{{\mathrm d} x} + h(\cdot), 
\end{equation}
since we have for all $\lambda \in {\mathbb C}$, 
\begin{eqnarray}\nonumber 
A \, v_{\lambda}(x) &=& \Big(g(x) \frac{{\mathrm d}}{{\mathrm d} x} + h(x)\Big) \, {\mathrm e}^{H(x) + \lambda G(x)} \\
\label{gx}
&=& g(x) \, {\mathrm e}^{H(x) + \lambda G(x)}\, \frac{(-h(x)+ \lambda)}{g(x)} + h(x) \, {\mathrm e}^{H(x) + \lambda G(x)} = \lambda \, v_{\lambda}(x). 
\end{eqnarray}
We can therefore apply the  generalized Prony method to recover (\ref{genex}), and with the operator $A$ in (\ref{4.1.1}) this leads to a recovery scheme  that involves the samples 
$$ F \Big( \Big(g(\cdot) \, \frac{{\mathrm d}}{{\mathrm d} x} + h(\cdot)\Big)^{k } f \Big), \qquad k=0, \ldots , 2M-1.$$
However, these samples  may be difficult to provide.

We therefore apply the GOP approach  with $\varphi(z) = \exp(\tau z)$. For $f$ of the form (\ref{genex})
it follows that 
\begin{eqnarray}
\nonumber
{\mathrm e}^{\tau A} f(x) &=& {\mathrm e}^{\tau (g(\cdot) \, \frac{{\mathrm d}}{{\mathrm d} x} + h(\cdot))} f(x) =
\sum_{\ell=0}^{\infty} \frac{\tau^{\ell}}{\ell!} \, \Big( g(\cdot) \,  \frac{{\mathrm d}}{{\mathrm d} x} + h(\cdot)\Big)^{\ell} \Big( \sum_{j=1}^{M} c_{j} \, {\mathrm e}^{H(x) + \lambda_{j} G(x)}\Big) \\
\nonumber
&=& \sum_{j=1}^{M} c_{j} \, \Big( \sum_{\ell=0}^{\infty} \frac{\tau^{\ell}}{\ell!} \, \lambda_{j}^{\ell} \Big) \, {\mathrm e}^{H(x) + \lambda_{j} G(x)} 
= \sum_{j=1}^{M} c_{j} \, {\mathrm e}^{\lambda_{j} \tau} \, {\mathrm e}^{H(x) + \lambda_{j} G(x)} \\
\nonumber
&=& {\mathrm e}^{H(x) - H(G^{-1}(\tau + G(x)))} \sum_{j=1}^{M} c_{j} \, {\mathrm e}^{H(G^{-1}(\tau + G(x)))+\lambda_{j} G(G^{-1}(\tau + G(x)))} \\
\label{4.go}
&=& {\mathrm e}^{H(x) - H(G^{-1}(\tau + G(x)))} \, f(G^{-1}(\tau + G(x))). 
\end{eqnarray}
Thus, the iteration operator $\varphi(A)$ of $A$  is the generalized shift operator $S_{G,H,\tau}: C({\mathbb R}) \to C({\mathbb R}) $ with 
\begin{equation}\label{shift1} S_{G, H, \tau}f (x) := \varphi(A) f(x) = {\mathrm e}^{\tau A} f (x) = {\mathrm e}^{H(x) - H(G^{-1}(\tau + G(x)))} \, f(G^{-1}(\tau + G(x))).
\end{equation}
This observation enables us to reconstruct $f$ in (\ref{genex})  using  function values  instead of derivative values.

\begin{theorem} \label{theo4.1}
Let $G: I \to J \subset {\mathbb R}$ be in $C^{\infty}(I)$ with $|G'(x)| > 0$ for all $x \in I$, and  $H \in C^{\infty}(I)$. 
Further,  for some fixed $x_{0} \in I$ and $0 < |\tau| \le \pi/C$ let $\tau k + G(x_{0}) \in G(I)$ for $k=0, \ldots , 2M-1$, where 
$G(I):=\{ g(x): \, x \in I \}$ denotes the image of $G$.
Then $f$ in $(\ref{genex})$ with $|\textrm{Im}\, \lambda_{j} | \le C$  can be uniquely reconstructed  from the function samples $f(G^{-1}(\tau k + G(x_{0})))$, $k=0, \ldots , 2M-1$. 
\end{theorem}

\begin{proof}
Taking the differential operator $A$ in (\ref{4.1.1}) with $g$ and $h$ as in (\ref{gh}), it follows from (\ref{gx}) that ${\mathrm e}^{H(x) + \lambda_{j} G(x)}$ are eigenfunctions of $A$ to the pairwise distinct eigenvalues $\lambda_{j}$.  
As shown in (\ref{4.go}), we can apply $\varphi(z)=\exp(\tau z)$ and obtain the generalized shift operator $\varphi(A) =S_{G,H, \tau}$ in (\ref{shift1}).
One important consequence of the computations in (\ref{4.go})  is the observation that also
$$ \varphi(A)^{k} f = {\mathrm e}^{ \tau \, k A} f = {\exp} \left( \tau \, k \,  \Big(g(\cdot) \, \frac{{\mathrm d}}{{\mathrm d} x} + h(\cdot) \Big) \right) f = S_{G, H, k\tau}\,  f $$
holds. 
Therefore, we have $S_{G,H,\tau}^{k}= S_{G,H, k\tau}$, see also \cite{PSK18} for a different proof.
We apply now Theorem \ref{GOProM} to $f$ in (\ref{genex}) with the operator $\varphi(A) = S_{G,H,\tau}$, the point evaluation functional $F (f)= f(x_{0})$, and with $F_{k} (f) := F(\varphi(A)^{k} f)$.
By Theorem \ref{th_spectral_mapping}, the eigenfunctions
 ${\mathrm e}^{H(x) + \lambda_{j} G(x)}$  of $A= g(\cdot) \frac{{\mathrm d}}{{\mathrm d} x} + h(\cdot)$  to the eigenvalues $\lambda_{j}$ are also eigenfunctions of $S_{G, H,\tau}$, now to the eigenvalues ${\mathrm e}^{\lambda_{j} \tau}$. We only need to pay attention that these new eigenvalues are pairwise distinct. Since $\lambda_{j} \in {\mathbb R} + {\mathrm i} [-C,C)$, this is satisfied if $0 < \tau \le \frac{\pi}{C}$.
Therefore the mapping from ${\mathrm e}^{\lambda_{j} \tau}$ to $v_{\lambda_{j}} = {\mathrm e}^{H(\cdot) + \lambda_{j} G(\cdot)}$ is bijective.
Finally, $F (v_{\lambda_{j}})= v_{\lambda_{j}}(x_{0}) =  {\mathrm e}^{H(x_{0}) + \lambda_{j} G(x_{0})} \neq 0$. Hence, the sampling matrix 
\begin{eqnarray*}
 & & (F(\varphi(A)^{k+\ell}f))_{k,\ell=0}^{M-1,M} = ((S_{G,H,\tau(k+ \ell)} f)(x_{0}))_{k,\ell=0}^{M-1,M} \\
&=& \left({\mathrm e}^{H(x) - H(G^{-1}(\tau(k+ \ell) + G(x_{0})))} \, f(G^{-1}(\tau (k+\ell)+ G(x_{0}))\right)_{k,\ell=0}^{M-1,M}
\end{eqnarray*}
is admissible by Lemma \ref{lem:1} and is already determined by 
the well-defined sampling values  $f(G^{-1}(\tau k + G(x_{0})))$, $k=0, \ldots , 2M-1$. Thus, Theorem \ref{GOProM} can be applied and the assertion follows.
\end{proof}

\begin{remark}
If the generalized shift operator $S_{G,H,\tau}$ is used to recover the expansion $f$ in $(\ref{genex})$, then the assumptions on $G$ and $H$ can be relaxed. It is sufficient to have continuous functions $G$ and $H$, where $G$ is monotone on $I$.
\end{remark}

\begin{example}
We want to recover  an expansion of the form 
\begin{equation}\label{excos} f(x) = \sum_{j=1}^{M} c_{j} \, {\mathrm e}^{\lambda_{j} \cos(x)}
\end{equation}
and have to find the parameters $c_{j} \in {\mathbb C} \setminus \{ 0 \}$ and $\lambda_{j} \in {\mathbb R} + {\mathrm i} [-\pi, \, \pi)$ by employing Theorem \ref{theo4.1}.
We take $G(x) := \cos(x)$ which is  monotone on $[0,\pi]$, i.e., we can choose $I=[0,\pi]$ and $G(I)=[-1, \, 1]$. 
Then, $G: I \to G(I)$ is bijective, and  $G^{-1}(x) = \arccos(x)$ is well-defined  as a function from $G(I)$ onto $I$.  Further, let $H(x) :=0$.  
Taking $g(x) := \frac{1}{G'(x)} =\frac{-1}{\sin x}$   and $h(x) := 0$,   we conclude that the functions ${\mathrm e}^{\lambda_{j} \cos(x)}$ in the expansion (\ref{excos}) are eigenfunctions of the differential operator $A = - \frac{1}{\sin(x)} \, \frac{\mathrm d}{{\mathrm d} x}$. We apply $\varphi(z) = \exp(\tau z)$ and obtain
the generalized shift operator of the form 
$$ \varphi(A) f(x) = S_{\cos, 0, \tau} f (x) = f(\arccos(\tau + \cos(x))). $$
We choose $x_{0}=0$, i.e., $G(x_{0}) = 1$, and $\tau= -\frac{1}{M}$ such that the values $\cos(x_{0})+ k \tau = 1 - k/M \in G(I)$ for $0 \ldots , 2M-1$. Thus 
$$ S_{\cos, 0, \tau}^{k}f (x_{0}) = S_{\cos,0,  k\tau}f(0) = f(\arccos(k\tau + 1)), \qquad k=0, \ldots , 2M-1,$$
are well-defined.
According to Theorem \ref{theo4.1}, $f(x)$ in (\ref{excos}) is already completely described by these values.
In this case, ${\mathrm e}^{\lambda_{j} \cos(x)}$ are eigenfunctions to $S_{\cos, 0,  \tau}$ corresponding to the eigenvalues ${\mathrm e}^{{\lambda_{j} \tau}}$. Therefore, defining the Prony polynomial
$$ P_{\cos}(z) = \prod_{j=1}^{M} ( z - {\mathrm e}^{\lambda_{j} \tau} ) = \sum_{\ell=0}^{M} p_{\ell} \, z^{\ell} $$
we find with (\ref{excos})
\begin{eqnarray*}
\sum_{\ell=0}^{M} p_{\ell} \, f(\arccos(1+(m+\ell)\tau)) &=& \sum_{\ell=0}^{M} p_{\ell} \, \sum_{j=1}^{M} c_{j} \, {\mathrm e}^{\lambda_{j}(\cos(\arccos(1+(m+\ell)\tau)))}\\
&=& \sum_{j=1}^{M} c_{j} {\mathrm e}^{\lambda_{j} (1+m \tau)} \sum_{\ell=0}^{M} p_{\ell} \, {\mathrm e}^{\lambda_{j} \ell \tau} =0
\end{eqnarray*}
for $m=0, \ldots , M-1$. This homogeneous  linear system provides  the coefficients $p_{0}$, $\ldots$, $p_{M-1}$, and $p_{M}=1$ of $P_{\cos}(z)$. Having found $P_{\cos}(z)$, we can extract its zeros ${\mathrm e}^{\lambda_{j} \tau}$, recover $\lambda_{j}$ and finally find $c_{j}$ by solving a linear system for the given function values. $\Box$
\end{example}

\begin{example}
We want to recover an expansion into shifted Gaussians of the form
\begin{equation}\label{gauss}
 f(x) = \sum_{j=1}^{M} c_{j} \, {\mathrm e}^{-\alpha(x-\lambda_{j})^{2}}, 
 \end{equation}
where we assume that $\alpha \in {\mathbb R} \setminus \{ 0 \}$ is given beforehand, and we need to reconstruct $c_{j} \in {\mathbb C}\setminus \{ 0\}$
and $\lambda_{j} \in {\mathbb R}$, $j=1, \ldots , M$. 
 By direct comparison we have
$ {\mathrm e}^{-\alpha(x-\lambda_{j})^{2}} = {\mathrm e}^{\lambda_{j}^{2}}\, {\mathrm e}^{H(x) + \lambda_{j} G(x)}$
with 
$$ H(x) = -\alpha x^{2} , \qquad G(x) = 2 \alpha x, $$
and with the linear factor ${\mathrm e}^{\lambda_{j}^{2}}$.
Thus, taking $g(x) := 1/G'(x) = 1/(2 \alpha) $ and $h(x) := H'(x)/G'(x) = -x$, it follows that  $v_{\lambda_{j}}(x) =  {\mathrm e}^{-\alpha(x-\lambda_{j})^{2}}$ satisfies the differential equation 
$$  \left( \frac{1}{2\alpha} \frac{{\mathrm d}}{{\mathrm d} x}  - x  \right) \, v_{\lambda_{j}}(x) = \lambda \, v_{\lambda_{j}}(x), $$
i.e., ${\mathrm e}^{\alpha(x- \lambda_{j})^{2}}$ are eigenfunctions of the operator $A$ in (\ref{4.1.1}) with $g$ and $h$ as above. 
According to Theorem \ref{theo4.1} we can therefore recover the expansion into shifted Gaussians in (\ref{gauss}) using the function samples 
$$ f(G^{-1}(k \tau + G(0)) = f \Big(\frac{\tau}{2 \alpha} k \Big), \qquad k=0, \ldots , 2M-1,$$ 
where we have taken $x_{0}=0$ and arbitrary  real step size $\tau \neq 0$, since $G(x)$ is monotone on ${\mathbb R}$ and the eigenvalues ${\mathrm e}^{\lambda_{j} \tau}$ are real, see also \cite{PSK18}, Section 4.1.   $\Box$
 \end{example}
 
\begin{remark}
We mention that there are other approaches to recover expansions into shifted Gaussians, see e.g. \cite{Vetterli.2002}.
When one is interested in approximation of functions by sparse sums of the form (\ref{gauss}), the question occurs, whether arbitrarily narrow Gauss pulses be constructed by linearly combining arbitrarily wider Gauss pulses.
This question has been recently discussed in \cite{FP14}.
\end{remark}

The approach to consider  eigenfunctions of the  form $v_{\lambda}(x) = {\mathrm e}^{H(x) + \lambda G(x)}$  for differentiable functions $G(x)$ and $H(x)$, where $G(x)$ is strictly monotone on some interval $I$ opens the way  to recover many different expansions  of the form (\ref{genex})  using only special function values of $f$. In Table 1, we summarize  some examples for $g(x)$, $G(x)$, and arbitrary $H(x)$ (resp. $h(x)$), the corresponding  eigenfunctions $v_{\lambda}$ as well as the needed function samples for GOP.

\begin{table}[ht]
  \centering
    \begin{tabular}{| c | c | l | l | l | l |}
    \hline
    $g(x)$  & $G(x)$ & eigenfunctions $v_{\lambda}$ &
    sampling values\\[0.1cm]
    \hline&&&\\[-1em] $1/x$ & $-\frac{1}{2}x^2$ &
    $\exp(H(x)-\frac{\lambda}{2}x^2)$& $f\left(\sqrt{-k \tau + x_0}\right)$\\[0.1cm]
    \hline&&&\\[-1em] $1$ & $x$ &$\exp(H(x) + \lambda x)$&
    $f\left(k \tau +x_0\right)$\\[0.1cm]
    \hline&&&\\[-1em] $x$ & $\log(x)$ & $ {\mathrm e}^{H(x)} \, x^{\lambda}$ &
    $f\left(e^{k \tau}x_0\right)$\\[0.1cm]
    \hline&&&\\[-1em] $x^{p}$ $(p \neq 1)$ & $\frac{x^{1-p}}{1-p}$ & $ \exp(H(x)+ \lambda x^{1-p}/(1-p))$ &
    $f((1-p) \tau k + x_{0}^{1-p})^{1/1-p})$\\[0.1cm]    
    \hline&&&\\[-1em] $-\sqrt{1-x^2}$ & $\arccos(x)$ & $\exp(H(x) + \lambda\arccos(x))$
    & $f(\cos(k\tau+\arccos(x_0)))$\\[0.1cm]
    \hline&&&\\[-1em] $\sqrt{1-x^2}$ & $\arcsin(x)$ & $\exp(H(x) + \lambda\arcsin(x))$
    & $f\left(\sin(k\tau+\arcsin(x_0))\right)$\\[0.1cm]
    \hline&&&\\[-1em] $\sqrt{x^{2}-1}$ & $\textrm{arcosh} \, (x)$ & $ \exp (H(x) + \lambda \, \textrm{arcosh} \,  (x)) $ &
    $f\left(\cosh(k\tau  + \textrm{arcosh}(x_0))\right)$\\[0.1cm]
     \hline&&&\\[-1em] $\sqrt{x^{2}+1}$ & $\textrm{arsinh} \, (x)$ & $ \exp (H(x) + \lambda \, \textrm{arsinh} \,  (x)) $ &
    $f\left(\sinh(k\tau  + \textrm{arsinh}(x_0))\right)$\\[0.1cm]
     \hline&&&\\[-1em] $\frac{1}{\cos(x)}$ & $\sin(x)$ &
    $\exp(H(x) + \lambda\sin(x))$ & $f\left(\arcsin(k\tau+\sin({x}_{0}))\right)$\\[0.1cm]
    \hline&&&\\[-1em] $-\frac{1}{\sin(x)}$ & $\cos(x)$ & $\exp(H(x)+ \lambda\cos(x))$
    & $f\left(\arccos(k\tau+\cos({x_{0}}))\right)$\\[0.1cm]
    \hline&&&\\[-1em] $-\frac{1}{\cosh(x)}$ & $\sinh(x)$ & $\exp(H(x)+ \lambda\sinh (x))$
    & $f\left(\textrm{arsinh}(k\tau+\sinh({x_{0}}))\right)$\\[0.1cm]
\hline&&&\\[-1em] $-\frac{1}{\sinh(x)}$ & $\cosh(x)$ & $\exp(H(x)+ \lambda\cosh (x))$
    & $f\left(\textrm{arcosh}(k\tau+\cosh({x_{0}}))\right)$\\[0.1cm]

    \hline 
    \end{tabular}
    \begin{center}
  \caption{Examples of operators $A=g(\cdot) \frac{\mathrm d}{{\mathrm d} x} + h(\cdot)$, corresponding eigenfunctions} \textit{$v_{\lambda} = \exp(H(\cdot) + \lambda \, G(\cdot))$ and sampling values for $k=0, \ldots ,2M-1$ with\\ sampling parameter $\tau$ to recover expansions $f$ in $(\ref{genex})$.}
  \end{center}
  \label{tab:GEProM_Examples}
\end{table}

\subsection{Second order differential operators and generalized symmetric shifts}

We consider now the reconstruction problem  to find all parameters $c_{j} \in {\mathbb C} \setminus  \{ 0 \}$ and
$\lambda_{j} \in [0, C)$ of 
\begin{equation}\label{cos*}
f(x) = \sum_{j=1}^{M} c_{j} \, \cos( \lambda_{j} \, G(x)).
\end{equation}
As before, we assume that $G \in C^{\infty}(I)$ for some interval $I= [a, b] \subset {\mathbb R}$ and that $G'$ is strictly positive (or strictly negative) on $I$. Let $g(x):= 1/G'(x)$.
We consider now the special  differential operator of second order acting on $f(x)$ as follows
\begin{equation}\label{4.2A}
B f(x) := A^{2} f(x) = \left( \Big(g(\cdot) \frac{{\mathrm d} }{{\mathrm d} x} \Big)^{2} f \right)(x) = (g(x))^{2} f''(x) + g(x) \, g'(x) f'(x). 
\end{equation}
Similarly as in (\ref{gx}), we observe that the functions ${\mathrm e}^{{\mathrm i} \lambda G(x)}$ and ${\mathrm e}^{-{\mathrm i} \lambda G(x)}$ are  the two eigenfunctions of $B$ to the eigenvalue $-\lambda^{2}$.
Therefore, also $\cos (\lambda\, G(x))$ and $\sin (\lambda \, G(x))$ are eigenfunctions of $B$ to $-\lambda^{2}$.

In order to ensure that the map from eigenvalues to eigenfunctions $-\lambda^{2} \to v_{\lambda}$ is bijective, we restrict ourselves to the eigenfunctions $\cos(\lambda \, G(x))$  with $\lambda \ge 0$. 
\medskip

Then, the function $f$ in (\ref{cos*}) can be understood as an expansion into eigenfunctions $\cos( \lambda_{j} \, G(x))$ of the operator $B$ in (\ref{4.2A}), and 
according to the generalized Prony method in Theorem \ref{GPM}, we can reconstruct  $f$  using the values $F \Big( (g(\cdot)  \frac{{\mathrm d} }{{\mathrm d} x})^{2k} f \Big)$, $k=0, \ldots , 2M-1$ with some suitable functional $F: C^{\infty}(I) \to {\mathbb C}$.
\smallskip

We want to apply GOP to derive a simpler reconstruction scheme.
We take the analytic function $\varphi(z) = \cos(\tau z^{1/2})$ and obtain for $f$ in (\ref{cos*}) according to (\ref{4.go})
\begin{eqnarray*}
\varphi(B) f (x) &=& \varphi(A^{2}) f (x) = 
  \cos(\tau  A) f(x) \\
  &=& \frac{1}{2} \left[ \exp \Big(\tau g(\cdot) \frac{{\mathrm d} }{{\mathrm d} x} \Big) + \exp\Big(-\tau g(\cdot) \frac{{\mathrm d} }{{\mathrm d} x} \Big) \right] f (x)\\
&=& \frac{1}{2} \left[ f(G^{-1}(\tau + G(x))) + f(G^{-1}(-\tau + G(x))) \right].
\end{eqnarray*}
Thus, we find  here a symmetric generalized shift operator 
$$ S_{G,\tau}^{sym} f := \frac{1}{2} \left[ f(G^{-1}(\tau + G(\cdot))) + f(G^{-1}(-\tau + G(\cdot))) \right] $$
as an iteration operator of $B$, and $f$ in (\ref{cos*}) can also be  understood as a sparse  expansion into eigenfunctions of the operator $S_{G,\tau}^{sym}$ to the eigenvalues $\varphi(-\lambda_{j}^{2}) = \cos(\tau \lambda_{j})$. This observation 
enables us to reconstruct  $f$ in (\ref{cos*})  using only function values of $f$ instead of derivative values.
\medskip

\begin{theorem}\label{theo4.2}
Let $G: I \to J \subset {\mathbb R}$ be in $C^{\infty}(I)$ with $|G'(x)| > 0$ for all $x \in I$. 
Assume further, that for some fixed $x_{0} \in I$ we have $\cos (\lambda G(x_{0})) \neq 0$ for all $\lambda \in [0, C)$, and for a fixed $\tau$ with $0 < |\tau| \le \pi/C$ we have $\tau k + G(x_{0}) \in G(I)$ for $k=-2M+1, \ldots , 2M-1$.
Then the parameters $c_{j} \in {\mathbb C} \setminus \{ 0 \}$ and $\lambda_{j} \in [0, C)$, $j=1, \ldots, M$, of $f$ in $(\ref{cos*})$  can be uniquely reconstructed  from the samples $f(G^{-1}(\tau k + G(x_{0})))$, $k=-2M+1, \ldots , 2M-1$.
\end{theorem}
\begin{proof}
We  apply  Theorem \ref{GOProM}, where we use the operator $\varphi(B)=  \cos(\tau A) = S_{G,\tau}^{sym}$, the point evaluation functional $F f = f(x_{0})$, and the set of sampling functionals $F_{k} = F (\varphi(B)^{k})$, $k=0, \ldots , M-1$. 
From Theorem \ref{th_spectral_mapping} it follows that the eigenfunctions $\cos(\lambda_{j} G(x))$ of $B$ in (\ref{4.2A}) are also eigenfunctions of 
$S_{G,\tau}^{sym}$. Indeed, we find by direct computation
\begin{eqnarray*}
S_{G,\tau}^{sym} \, \cos(\lambda_{j} G(x)) &=& \frac{1}{2} \left[ \cos(\lambda_{j} G(G^{-1}(\tau + G(x)))) + \cos(\lambda_{j} G(G^{-1}(-\tau + G(x)))) \right] \\
&=& \frac{1}{2} \left[ \cos (\lambda_{j}(\tau + G(x))) + \cos (\lambda_{j}(-\tau + G(x))) \right ]\\
&=& \cos(\lambda_{j} \tau) \, \cos( \lambda_{j} G(x)).
\end{eqnarray*}
Therefore, the eigenvalues have here the form $\cos(\lambda_{j} \tau)$ and are pairwise different for $\lambda_{j} \in [0, C)$ if $0 < \tau < \frac{\pi}{C}$. Further, the  sampling matrix $(F_{k} (\varphi(B)^{\ell} f ))_{k,\ell=0}^{M-1,M}$ is admissible by Lemma \ref{lem:1}.
This sampling matrix has Hankel structure and is determined by 
$$ F_{k} (f) = F( (S_{G, \tau}^{sym})^{k} f) = ((S_{G, \tau}^{sym})^{k} f)(x_{0}) = \frac{1}{2^{k}} \sum_{r=0}^{k} \binom{k}{r} f(G^{-1}(G(x_{0}) + (k-2r)\tau)) $$
for $k=0, \ldots , 2M-1$. Thus the assertion follows.
\end{proof}

\begin{example}
We  want to  reconstruct expansions of the form 
\begin{equation}\label{cheb} f(x) = \sum_{j=1}^{M} c_{j} \, \cos(\lambda_{j} \, \arccos (x))
\end{equation}
with $c_{j} \in {\mathbb C} \setminus \{0 \}$ and $\lambda_{j} \in [0, C)$.
Therefore, we choose $G(x):=\arccos(x)$ on the interval $[-1,1]$, and $g(x) := 1/G'(x) = -(1-x^{2})^{1/2}$. According to our observations we take 
 $A f(x) = g(x) f'(x) = -\sqrt{1-x^{2}} \, f'(x)$ and 
$$ B f(x) = A^{2} f(x) = \left(\sqrt{1-(\cdot)^{2}} \, \frac{{\mathrm d} }{{\mathrm d} x} \right)^{2} f(x) = (1-x^{2}) \, f''(x) - x f(x) $$
on $I=[-1,1]$.
Then, $B$ possesses the eigenfunctions $\cos(\lambda \, \arccos x)$ for $\lambda \ge 0$. Taking  the non-negative integers $\lambda =n \in {\mathbb N}_{0}$, we particularly obtain  the Chebyshev polynomials  $T_{n}(x) = \cos(n \arccos x)$. According to Theorem \ref{theo4.2} we can now reconstruct the expansion (\ref{cheb})
using only the samples
$ ((S_{\arccos, \tau}^{sym})^{k} f)(x_{0})$, $k=0, \ldots , 2M-1$, which can  be computed from the values
$$ f( \cos( k \tau + \arccos(x_{0}))), \qquad k=-2M+1, \ldots , 2M-1. $$
We can choose $x_{0} = 1$ to ensure that 
$ \cos(\lambda \, G(x_{0})) = \cos(\lambda \, \arccos(1)) = 1 \neq 0 $
for all $\lambda \in [0,C)$. Further, we take
 $\tau \in (0, \min \{\frac{\pi}{C}, \, \frac{\pi}{2M} \})$ such that $k\tau + \arccos x_{0} = k \tau \in [0, \pi )$ for $k=0, \ldots , 2M-1$. In this special case the values $f(\cos(k \tau))$, $k=0, \ldots , 2M-1$, are sufficient for full recovery since the cosine function is symmetric.
Different approaches to recover expansions into Chebyshev polynomials are taken in \cite{Potts.2014} and \cite{PSK18}. $\Box$
\end{example}

\section{Generalized sampling for the Prony method} 
\label{sec_generalized-sampling}
\setcounter{equation}{0}

In this section we study  admissible sampling schemes in GOP in more detail and want to give some special applications.

Let us assume that the normed vector space $V$ is a subspace of $L^{2}([a, b])$ and fix the linear operator $A: V \to V$. 
We denote with $\sigma(A)$ a  fixed set of pairwise different eigenvalues of $A$ 
and consider the set $V_{\sigma}$ of corresponding eigenvectors such that the map $\lambda \to v_{\lambda}$ is a bijective map from $\sigma(A)$ onto $V_{\sigma}$. 
By Theorem \ref{GOProM} we know that $A$ can be replaced by an iteration operator $\varphi(A)$.

In this section  we will focus on  finding an admissible set $\{F_{k} \}_{k=0}^{M-1}$ of sampling functionals according to Definition \ref{def_evaluation_scheme} such that entries of the sampling matrix $(F_{k}(A^{\ell} f))_{k,\ell=0}^{M-1,M}$ can be simply computed.
We recall that a  set of sampling functionals $F_{k}: V \to {\mathbb C}$ is admissible if  $(F_{k}(v_{\lambda}))_{k=0,\lambda \in \Lambda_{M}}^{M-1}$ has full rank $M$ for all subsets $\Lambda_{M} \subset \sigma(A)$ with cardinality $M$.
Then it follows by Theorem \ref{theosamp} that the sampling matrix $(F_{k} (A^{\ell}f))_{k,\ell=0}^{M-1,M}$ has full rank $M$ for each $f \in {\mathcal M}(A)$ such that $f$ can be uniquely recovered.

We consider functionals  $F_{k}:{\mathcal M}(A) \to {\mathbb C}$ which can be written as 
\begin{equation}
\label{fk1} F_{k} (f) := \langle f, \, \phi_{k} \rangle = \int_{a}^{b} f(x) \, \phi_{k}(x) \, {\mathrm d} x,
\end{equation}
where $(a,b) \subseteq {\mathbb R}$ is a suitable interval and $\phi_{k}$ is
 some kernel function or distribution, such that the integral in (\ref{fk1}) is well-defined in a distribution sense. For example, we can take $\phi_{k}$ to be the $\delta$-distribution, 
$$ F_{k} (f) := \langle f, \, \delta(\cdot - x_{0}) \rangle = \int_{a}^{b} f(x) \, \delta(\cdot - x_{0}) \, {\mathrm d} x = f(x_{0}), \qquad x_{0} \in [a,b]. $$
Using the adjoint operator, the entries of the sampling matrix can be written as
\begin{equation}\label{sampi}  F_{k} (A^{\ell} f) = \langle A^{\ell} f, \, \phi_{k} \rangle = \langle
f, \, (A^{*})^{\ell} \phi_{k} \rangle = \int_{a}^{b} f(x) \, (A^{*})^{\ell} \, \phi_{k}(x) \, {\mathrm d} x.
\end{equation}
If $A$ is a linear differential operator, the consideration of powers of the adjoint operator $A^{*}$ applied to $\phi_{k}$ is particularly useful, if we  cannot acquire  derivative samples of $f$ but special moments instead.
In this case, we need to assume that the kernel functions $\phi_{k}$ are sufficiently smooth on $[a,b]$, such that $(A^{*})^{\ell} \phi_{k} \in L^{2}([a,b])$.
For admissibility we need now to ensure that  $(\langle v_{\lambda}, \, \phi_{k} \rangle)_{k=0,\lambda \in \Lambda_{M}}^{M-1}$ has full rank $M$.


\begin{example} 
We consider again the example of exponential sums to present the variety of possible sampling matrices that can be used.
Let 
$$ f(x) = \sum_{j=1}^{M} c_{j} \, {\mathrm e}^{T_{j} x}$$
with $c_{j} \in {\mathbb C} \setminus \{ 0 \}$, $T_{j} \in {\mathbb R} + {\mathrm i} [-\pi, \, \pi)$, where ${\mathrm e}^{T_{j}x}$ are eigenfunctions of $A=\frac{\mathrm d}{{\mathrm d} x}$ to the eigenvalue $T_{j}$.
Here ${\mathcal M}(A)$ is a subset of the Schwartz space, and thus obviously a subspace of $L^{2}([a,b])$ for each interval $[a,b]$ and also of $L^{2}({\mathbb R})$.
We present a variety of sampling schemes which are  all admissible and of the form (\ref{sampi}).

a) Let $F (f) := \int_{-\infty}^{\infty} f(x) \, \delta(x - x_{0}) \, {\mathrm d} x = f(x_{0})$ be the point evaluation  functional 
with $x_{0} \in {\mathbb R}$ and let $F_{k}(f) := F(A^{k}f)$.
Then the entries of the sampling matrix are of the form 
$$ F_{k}(A^{\ell} f) = \int_{-\infty}^{\infty} A^{\ell }f (x) \, A^{k} \delta(x - x_{0}) \, {\mathrm d} x = \int_{-\infty}^{\infty} f (x) \, \delta^{(k+\ell)}(x - x_{0})  {\mathrm d} x = f^{(k+\ell)}(x_{0}) $$
used in Section \ref{subsecdif}, where we need derivative values $f^{(k)}(x_{0})$, $k=0, \ldots , 2M-1$. 
The used kernel functions are in this case the distributions $\phi_{k} = A^{k} \delta(\cdot - x_{0}) = \delta^{(k)}(\cdot - x_{0})$, i.e., derivatives of the Delta distribution.
Admissibility is ensured  since  for any $T_{j} \in {\mathbb R} + {\mathrm i} [-\pi, \, \pi)$,  
$$(\langle {\mathrm e}^{T_{j}\cdot}, \, \phi_{k} \rangle)_{k=0, j=1}^{M-1,M} = ( T_{j}^{k} {\mathrm e}^{T_{j}x_{0}})_{k=0,j=1}^{M-1,M} = (T_{j}^{k})_{k=0,j=1}^{M-1,M} \, \textrm{diag} ( {\mathrm e}^{T_{j}x_{0}})_{j=1}^{M}
$$
has full rank $M$.
\smallskip

b) By Lemma \ref{lem:1} we can also take $F_{k}(f) = F(\psi(A)^{k} f)$ for some iteration operator $\psi(A)$ with $F$ as in a). With 
$\psi(A) = \exp(\tau A) = S_{\tau}$, $\tau \neq 0$, see Example \ref{exs}, we obtain the admissible sampling matrix with entries
\begin{eqnarray*}
 F_{k}(A^{\ell} f) &=& \int_{-\infty}^{\infty} (S_{\tau}^{k} A^{\ell} f)(x) \, \delta(x-x_{0})\, {\mathrm d} x\\
 & =& \int_{-\infty}^{\infty}  f(x) \, ( (A^{\ell})^{*} (S_{\tau}^{k})^{*} \delta) (x-x_{0})\, {\mathrm d} x\\
 &=& \int_{-\infty}^{\infty}  f(x) \,  \delta^{(\ell)} (x-\tau k-x_{0})\, {\mathrm d} x =  f^{(\ell)}(x_{0}+ \tau k), 
 \end{eqnarray*}
where we need the values $f^{(\ell)}(x_{0} + k \tau)$, $\ell=0, \ldots , M$, $k=0, \ldots , M-1$, see Section \ref{subsec:gen2}.
We have here $\phi_{k} = (S_{\tau}^{k})^{*} \delta(\cdot - x_{0}) = \delta(\cdot - \tau k -x_{0})$, $k=0, \ldots, M-1$.
\smallskip

c) Consider now the functional 
\begin{equation} \label{fun1} F (f) := \int_{0}^{1} f(x) \, \phi(x) \, {\mathrm d} x
\end{equation}
with $\phi(x) := x^{2M}(1-x)^{2M}$. Then 
$$F ( {\mathrm e}^{T \cdot}) = \int_{0}^{1} {\mathrm e}^{Tx} \, \phi(x) {\mathrm d} x \neq 0$$ 
for all $T \in {\mathbb R} + {\mathrm i} [-\pi, \, \pi)$ since $\phi(x)>0$ for $x \in (0,1)$.
Thus, with $F_{k} := F(A^{k})$ we obtain
\begin{eqnarray*}  
F_{k}(A^{\ell}f) &=& F(A^{k+\ell}f ) = \int_{0}^{1} f^{(k+\ell)}(x) \, x^{2M} (1-x)^{2M} \, {\mathrm d} x \\
&=& (-1)^{k+\ell} \, \int_{0}^{1} f(x) \, [x^{2M} (1-x)^{2M}]^{(k+\ell)} \, {\mathrm d} x
\end{eqnarray*}
is admissible. These values can be computed from the moments
$ \int_{0}^{1} f(x) x^{s} \, {\mathrm d} x$ for $s=0, \ldots , 4M$.
The functions $\phi_{k}$ are here defined as $\phi_{k} := \phi^{(k)} $, $k=0, \ldots , M-1$.
\smallskip

d) Let us now take the functional $F$ as in (\ref{fun1}), but with $\phi(x) := x^{M}(1-x)^{M}$ and let 
$F_{k}(f) := F(\exp(k A) f)  = F(S_{1}^{k} \, f)$ according to Lemma \ref{lem:1}. Then we get the entries of the admissible sampling matrix in the form 
\begin{eqnarray*} F_{k}(A^{\ell}f) \!\! &=& \!\! \int_{0}^{1} f^{(\ell)}(x+ k) \, x^{M} (1-x)^{M} \, {\mathrm d} x =
(-1)^{\ell} \int_{0}^{1} f(x+k) \, [x^{M} (1-x)^{M}]^{(\ell)} \, {\mathrm d} x \\
&=& (-1)^{\ell} \int_{k}^{k+1} f(x) \, [(x-k)^{M} (k+1-x)^{M}]^{(\ell)}\,  {\mathrm d} x.
\end{eqnarray*}
These entries can be computed from the moments $\int_{0}^{1} f(x+k) \, x^{s} {\mathrm d} x$ for $k=0, \ldots , M-1$ and $s=0, \ldots , 2M$.
The functions $\phi_{k}$ are of the form $\phi_{k}(x)= (x-k)^{M} (k+1-x)^{M}$, $k=0, \ldots , M-1$.

e) Besides all the sampling schemes above, we know from Section \ref{subsecshift}  that $f$ can be reconstructed using the $2M$ samples $f(x_{0}+ k\tau)$, $k=0, \ldots , 2M-1$, with $x_{0} \in {\mathbb R}$, $\tau \neq 0$.
This sampling scheme also follows from Theorem \ref{GOProM} by replacing $A$ by the iteration operator $\exp(\tau A) =S_{\tau}$.
The simple equidistant sampling is obtained by taking $F_{k}= F(S_{\tau}^{k})$ and the kernel function $\phi(x) = \delta(x-x_{0})$ as in a), such that 
$$ F_{k}((\exp(\tau A))^{\ell} f))= F(S^{k+\ell} f) = f(x_{0}+ (k+\ell)\tau). $$
The kernel functions $\phi_{k}$ are here $\phi_{k}= \phi(\cdot - \tau k)$, $k=0, \ldots , M-1$.
Taking instead $F_{k}= F(S_{2\tau}^{k})$ we arrive at  
$$ F_{k}(S_{\tau}^{\ell} f) = F(S_{2\tau}^{k} S_{\tau}^{\ell}f) = f( x_{0}+\tau(2k+\ell)), \quad k=0, \ldots , M-1, \, \ell=0, \ldots , M, $$
and also this sampling matrix is admissible by Lemma \ref{lem:1}. Here we have now $\phi_{k} = \phi(\cdot - 2\tau k)$, $k=0, \ldots , M-1$. $\Box$

\end{example}
\medskip

\noindent
Besides the well-known example of exponential sums, we can also find new sampling schemes for expansions into eigenfunctions of differential operators of higher order, where we need to acquire moments instead of derivative values.
This can be always achieved  by employing  suitable kernels $\phi_{k}$ and the adjoint operator representation in (\ref{sampi}).

Let us consider the linear differential operator 
\begin{equation}\label{adif}
A:=
\sum\limits_{n=0}^{d} g_n(\cdot) \,  \frac{{\mathrm d}^n}{{\mathrm d}  x^n}
\end{equation}
of order $d$ with sufficiently smooth functions $g_{n}$.
Further, let $\sigma(A)$ be a subset of pairwise distinct eigenvalues $\lambda$ of $A$ with corresponding eigenfunctions $v_{\lambda} \in L^{2}([a,b])$ such that we have a bijection $\lambda \to v_{\lambda}$.

\begin{lemma}\label{lem5.1}
Let $A$ be an operator in $(\ref{adif})$ with $g_{n} \in C^{d}([a,b])$ for $n=0, \ldots , d$, 
and let $F: L^{2}([a,b]) \to {\mathbb C}$ be a functional given by 
$F f = \langle f, \, \phi \rangle$, where $\phi \in C^{d}([a,b])$ and 
$$ \lim_{x \to a} \phi^{(\ell)}(x) = \lim_{x \to b} \phi^{(\ell)}(x) = 0, \qquad \ell=0, \ldots , d. $$
Then 
$$ F (Af) = \langle Af, \, \, \phi \rangle = \left\langle f, \sum_{n=0}^{d} (-1)^{r} \sum_{\ell=0}^{r} \binom{r}{\ell}  g_{n}^{(\ell)} \, \phi^{(r-\ell)} \right\rangle,
$$
where $g_{n}^{(\ell)}$ and $\phi^{(\ell)}$ denote the $\ell$-th derivative of $g_{n}$ and $\phi$, respectively.
\end{lemma}

\begin{proof}
The proof follows simply by partial integration, where the boundary terms vanish because of the assumption on $\phi$.
\end{proof}
\medskip

Thus, we can apply the sampling scheme arising from (\ref{sampi}) where we need to compute with  derivatives of the kernel functions instead of derivatives of $f$.

\begin{example}[Sparse Legendre Expansions]
We want to recover a sparse expansion into Legendre
polynomials of the form 
$$f(x) :=\sum\limits_{j=1}^M c_j \, P_{n_j}(x)$$
where $c_{j} \in {\mathbb C} \setminus \{ 0 \}$, and $n_{j} \in {\mathbb N}_{0}$
with $0 \le n_{1} < n_{2} < \ldots < n_{M}$. The Legendre polynomials $P_{n}$,
$n \in {\mathbb N}_{0}$ are eigenfunctions of the differential operator of
second order $$A f(x):=(x^2-1) \, f''(x) + 2x\,  f'(x), $$
and we have 
$$A \, P_n = n(n+1) \, P_n.$$
Employing a functional of the form 
$$ F(f):=\int_{a}^{b}f(x)\phi_P(x) \, dx, $$
with a smooth kernel $\phi_{P}$ satisfying $\phi_{P}(a)=\phi_{P}(b)=0$ and $\phi_{P}'(a)=\phi_{P}'(b)=0$,
it follows that 
$$ \int_{a}^{b} Af(x) \, \phi_{P}(x) \, dx = \int_{a}^{b} f(x) \, A \phi_{P}(x)
 \, dx. $$
We choose the kernel
\begin{equation} \label{phip}
\phi_P(x):=\begin{cases}
(x-a)^{4M}(x-b)^{4M}\,\exp\left(-\alpha(x-\beta_0)^2(x-\beta_1)^2\right) & x\in
[a,\, b], \\
0 & x\not\in [a,\, b].\\
\end{cases}
\end{equation}
Here, the parameters $\beta_0$ and
$\beta_1$ are chosen to be outside of the interval $[a,\,  b]$, and $\alpha \ge 0$. For $\alpha=0$, $\phi_{P}$ is a polynomial of degree $8M$.

Taking for example $[a,\, b] = [-1/2, \, 3/4]$, it follows that the functional $F$ satisfies the admissibility condition 
$F(P_{n}) \neq 0$ for all $n \in {\mathbb N}_{0}$. 
Therefore, the expansion $f$ can be recovered  from the $2M$ samples
$$ F(A^{k} f) = \int_{-1/2}^{3/4} f(x) \, A^{k}\phi_{P}(x) dx, \quad k=0, \ldots, 2M-1. $$

We consider a small computational example. We want to recover the parameters $c_{j}$ and $n_{j}$ of the expansion 
$$f(x)=\sum\limits_{j=1}^3 c_{j}\, P_{n_j}(x)$$
from the 6 samples $F(A^{k} f)$, $k=0, \ldots , 5$. The true  parameters are given in Table \ref{parameters_legendre}.
{\small
\begin{table}[ht]
\renewcommand*\arraystretch{1.2}
$$\setlength\arraycolsep{12pt}
 \begin{array}{|r|r|r|r|} 
 \hline 
 n_{j} & 1 &  4  & 9 \\
 \hline
 c_j & 1.703 &  3.193  & 3.710 \\ 
\hline
 \end{array}
$$
\caption{Active degrees $n_{j}$ and the corresponding linear coefficients $c_j$ of $f$ with parameters \\ in Table \ref{parameters_legendre}.}
\label{parameters_legendre}
\end{table}
}

The signal with this parameters is presented in Figure \ref{Legendre_signal}.
\begin{figure}[ht] 
	\centering
  \includegraphics[width=0.6\textwidth]{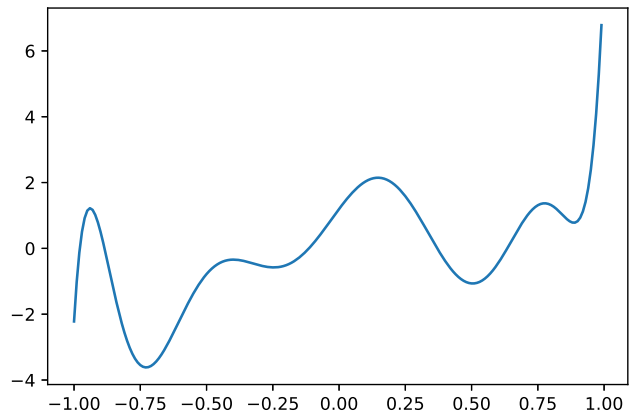}
	\caption{3-sparse Legendre expansion $f$ with parameters in Table \ref{parameters_legendre}.} 
		\label{Legendre_signal}
\end{figure}

We choose now the sampling kernel $\phi_{P}$ in (\ref{phip}) with $a=-1/2$, $b=3/4$, 
 $\alpha=0.1$, and
$-\beta_0=\beta_1=2$. 
The kernels $A^{k} \phi_{P}$, $k=0, \ldots , 5$, are depicted in Figure 2.
\relax
\begin{figure}[ht] 
	\centering
  \includegraphics[width=150mm]{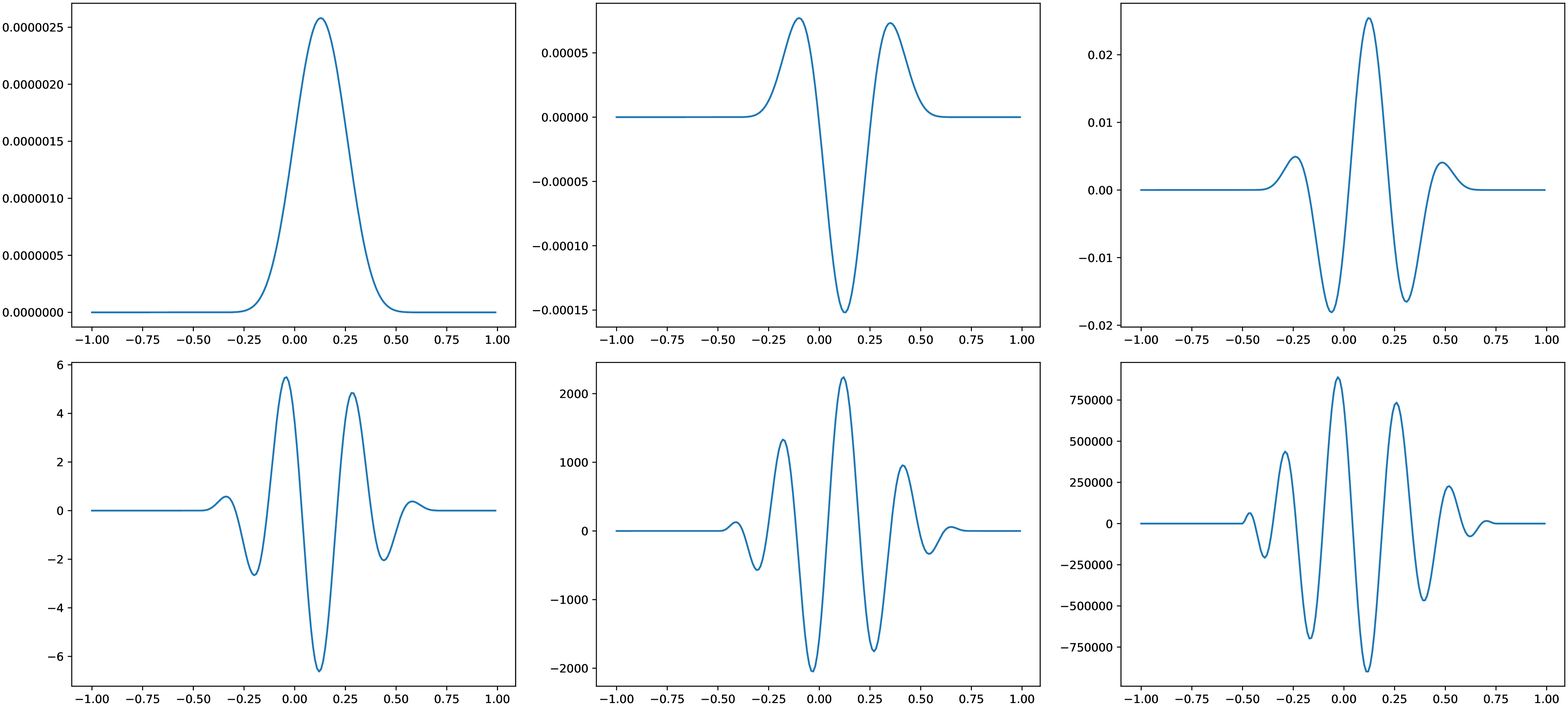}
  \centering
	{\small \textbf{Figure 2} 
	\textit{Sampling kernels $A^{k}\phi_{P}$, $k=0,1,2$ (first row) $k=3,4,5$ (second row) \\ for a 3-sparse Legendre expansion.}} 
		\label{Legendre_kernels}
\end{figure}
These kernels can now be used for any  $3-$sparse linear combination of
arbitrary Legendre polynomials. 
For our example, the sampling matrix has the form
$$
\left[
\begin{matrix}
F(f) & F(Af)) & F(A^{2}f) & F(A^{3}f)\\
F(Af) & F(A^{2}f)) & F(A^{3}f) & F(A^{4}f)\\
F(A^{2}f) & F(A^{3}f)) & F(A^{4}f) & F(A^{5}f)\\
\end{matrix}\right].
$$
The reconstructed parameters can be seen in Table \ref{rec_parameters_legendre}.
{\small
\begin{table}[ht]
\renewcommand*\arraystretch{1.2}
$$\setlength\arraycolsep{12pt}
 \begin{array}{|r|r|r|r|} 
 \hline 
 n_{j} & 1.00008823 & 4.00001099 & 9.00000026 \\
 \hline
 c_j & 1.703 &  3.193  & 3.710 \\

\hline
 \end{array}
$$
\caption{Computed parameters  $n_{j}$ and  $c_j$ for $f$.}
\label{rec_parameters_legendre}
\end{table}
}
The polynomial degrees are correctly recovered up to small rounding errors.
We round to the closest integer and get the exact values $n_{j}$.
The  coefficients $c_{j}$ are found using a $3 \times 3$ Vandermonde system.
Alternatively, to  recover the coefficients, we can use the orthogonality of Legendre polynomials and obtain 
$$ c_{j} = \frac{2 n_{j}+1}{2} \int_{-1}^{1} f(x) P_{n_{j}}(x) dx. $$
The numerical instabilities due to the exponentially growing functions $A^{k}\phi_{P}$ are an issue in this approach.
A clever choice of the parameters of $\phi$ can help to control the
amplitudes of $A^{k}\phi_{P}$. Another way is to apply a set of different functionals $F_{k}$ as proposed in Section 3.3.
\end{example}

\subsection*{Acknowledgement}
The authors gratefully acknowledge support by the German Research Foundation in the framework of the RTG 2088 and in the project PL 170/16-1.

\small
\bibliography{myBib}{}

\begin{thebibliography}{10}

\bibitem{ACH11}
F.~Andersson, M.~Carlsson, and M.V de~Hoop.
\newblock Sparse approximation of functions using sums of exponentials and
  {AAK} theory.
\newblock {\em J. Approx. Theory}, 163:213--248, 2011.

\bibitem{BSBV17}
G.~Baechler, A.~Scholefield, L.~Baboulaz, and M.~Vetterli.
\newblock Sampling and exact reconstruction of pulses with variable width.
\newblock {\em IEEE Trans. Signal Process.}, 65(10):2629--2644, 2017.

\bibitem{BY13}
D.~Batenkov and Y.~Yomdin.
\newblock On the accuracy of solving confluent {P}rony systems.
\newblock {\em SIAM J. Appl. Math.}, 73(1):134--154, 2013.

\bibitem{BT88}
M.~Ben-Or and P.~Tiwari.
\newblock A deterministic algorithm for sparse multivariate polynomial
  interpolation.
\newblock Proc. Twentieth Annual ACM Symp. Theory Comput., pages 301--309. ACM
  Press, New York, 1988.

\bibitem{BDB10}
J.~Berent, P.L. Dragotti, and T.~Blu.
\newblock Sampling piecewise sinusoidal signals with finite rate of innovation
  methods.
\newblock {\em IEEE Trans. Signal Process.}, 58(2):613--625, 2010.

\bibitem{BM05}
G.~Beylkin and L.~Monz\'{o}n.
\newblock On approximation of functions by exponential sums.
\newblock {\em Appl. Comput. Harmon. Anal.}, 19:17--48, 2005.

\bibitem{BE19}
A.~Bhandari and Y.~C. Eldar.
\newblock Sampling and super resolution of sparse signals beyond the fourier
  domain.
\newblock {\em IEEE Trans. Signal Process.}, 67(6):1508--1521, 2019.

\bibitem{BM86}
Y.~Bresler and A.~Macovski.
\newblock Exact maximum likelihood parameter estimation of superimposed
  exponential signals in noise.
\newblock {\em IEEE Trans. Acoust., Speech, Signal Process.}, 34(5):1081--1089,
  1986.

\bibitem{CL18}
A.~Cuyt and W.-s. Lee.
\newblock How to get high resolution results from sparse and coarsely sampled
  data.
\newblock {\em Appl. Comput. Harmon. Anal., online first}, 2018.

\bibitem{Dattoli.2000}
G.~Dattoli and D.~Levi.
\newblock Exponential operators and generalized difference equations.
\newblock {\em Nuovo Cimento Soc. Ital. Fis. B (12)}, 115(6):653--662, 2000.

\bibitem{Dattoli.1997}
G.~Dattoli, P.L. Ottaviani, A.~Torre, and L.~V\'azquez.
\newblock Evolution operators equations: Integration with algebraic and finite
  difference methods. applications to physical problems in classical and
  quantum mechanics and quantum field theory.
\newblock {\em Riv. Nuovo Cimento Soc. Ital. Fis. (4)}, 20(2):1--133, 1997.

\bibitem{Dragotti.2007}
P.L. Dragotti, M.~Vetterli, and T.~Blu.
\newblock Sampling moments and reconstructing signals of finite rate of
  innovation: {S}hannon meets {S}trang--{F}ix.
\newblock {\em IEEE Trans Signal Process.}, 55(5):1741--1757, 2007.

\bibitem{FP14}
P.~J. S.~G. Ferreira and A.~J. Pinho.
\newblock The natural scale of signals: Pulse duration and superoscillations.
\newblock International Conference on Acoustics, Speech and Signal Processing
  (ICASSP), pages 4176--4179. IEEE, 2014.

\bibitem{Hua.1991}
Y.~Hua and T.K. Sarkar.
\newblock On {SVD} for estimating generalized eigenvalues of singular matrix
  pencil in noise.
\newblock {\em IEEE Trans. Signal Process.}, 39(4):892--900, 1991.

\bibitem{Kailath.1990}
T.~Kailath.
\newblock {ESPRIT}--estimation of signal parameters via rotational invariance
  techniques.
\newblock {\em Optical Engineering}, 29(4):296, 1990.

\bibitem{LiS00}
L.~Li and T.~P. Speed.
\newblock Parametric deconvolution of positive spike trains.
\newblock {\em The Annals of Statistics}, 28(5):1279--1301, 2000.

\bibitem{Mourrain.2017}
B.~Mourrain.
\newblock Polynomial--exponential decomposition from moments.
\newblock {\em Found. Comput. Math.}, 82(3):339, 2017.

\bibitem{Osborne.1995}
M.R. Osborne and G.K. Smyth.
\newblock A modified {P}rony algorithm for exponential function fitting.
\newblock {\em SIAM J. Sci. Comput.}, 16(1):119--138, 1995.

\bibitem{Peter.2013b}
T.~Peter and G.~Plonka.
\newblock A generalized {P}rony method for reconstruction of sparse sums of
  eigenfunctions of linear operators.
\newblock {\em Inverse Problems}, 29(2), 2013.

\bibitem{PPR13}
T.~Peter, G.~Plonka, and D.~Ro{\c{s}}ca.
\newblock Representation of sparse {L}egendre expansions.
\newblock {\em J. Symbolic Comput.}, 50:159--169, 2013.

\bibitem{PPT11}
T.~Peter, D.~Potts, and M.~Tasche.
\newblock Nonlinear approximation by sums of exponentials and translates.
\newblock {\em SIAM J. Sci. Comput.}, 33(4):1920--1947, 2011.

\bibitem{Pisarenko.1973}
V.~F. Pisarenko.
\newblock The retrieval of harmonics from a covariance function.
\newblock {\em Geophys. J. Int.}, 33(3):347--366, 1973.

\bibitem{PP16}
G.~Plonka and V.~Pototskaia.
\newblock Application of the {AAK} theory for sparse approximation of
  exponential sums.
\newblock http://arxiv.org/pdf/1609.09603, 2016.

\bibitem{PP19}
G.~Plonka and V.~Pototskaia.
\newblock Computation of adaptive fourier series by sparse approximation of
  exponential sums.
\newblock {\em J. Fourier Anal. Appl.}, 25(4):1580--1608, 2019.

\bibitem{PPST18}
G.~Plonka, D.~Potts, G.~Steidl, and M.~Tasche.
\newblock {\em Numerical Fourier Analysis}.
\newblock {Birkh\"auser}, 2018.

\bibitem{PSK18}
G.~Plonka, K.~Stampfer, and I.~Keller.
\newblock Reconstruction of stationary and non-stationary signals by the
  generalized {P}rony method.
\newblock {\em Anal. and Appl.}, 17(2):179--210, 2019.

\bibitem{PT14}
G.~Plonka and M.~Tasche.
\newblock {P}rony methods for recovery of structured functions.
\newblock {\em GAMM Mitt.}, 37(2):239--258, 2014.

\bibitem{PW13}
G.~Plonka and M.~Wischerhoff.
\newblock How many {F}ourier samples are needed for real function
  reconstruction?
\newblock {\em J. Appl. Math. and Comput.}, 42(1-2):117--137, 2013.

\bibitem{Potts.2010}
D.~Potts and M.~Tasche.
\newblock Parameter estimation for exponential sums by approximate {P}rony
  method.
\newblock {\em Signal Process.}, 90(5):1631--1642, 2010.

\bibitem{Potts.2014}
D.~Potts and M.~Tasche.
\newblock Sparse polynomial interpolation in {C}hebyshev bases.
\newblock {\em Linear Algebra Appl.}, 441:61--87, 2014.

\bibitem{Schmidt.1986}
R.O. Schmidt.
\newblock Multiple emitter location and signal parameter estimation.
\newblock {\em IEEE Trans. On Antennas and Propagation}, 34(3):276--280, 1986.

\bibitem{Seelamantula.2015}
C.S. Seelamantula.
\newblock Opera: Operator-based annihilation for finite-rate-of-innovation
  signal sampling.
\newblock In G.~Pfander, editor, {\em Sampling {T}heory, a {R}enaissance},
  Applied and Numerical Harmonic Analysis, pages 461--484. Birkh{\"a}user,
  Cham, 2015.

\bibitem{Urigen.2013}
J.A. Urigen, T.~Blu, and P.L. Dragotti.
\newblock {FRI} sampling with arbitrary kernels.
\newblock {\em IEEE Trans. Signal Process.}, 61(21):5310--5323, 2013.

\bibitem{Vetterli.2002}
M.~Vetterli, P.~Marziliano, and T.~Blu.
\newblock Sampling signals with finite rate of innovation.
\newblock {\em IEEE Trans. Signal Process.}, 50(6):1417--1428, 2002.

\bibitem{ZP18}
R.~Zhang and G.~Plonka.
\newblock Optimal approximation with exponential sums by a maximum likelihood
  modification of {P}rony’s method.
\newblock {\em Adv. Comput. Math.}, 45(3):1657--1687, 2019.

\end{thebibliography}
\bibliographystyle{plain}

\end{document}